\documentclass[12pt]{amsart}

\usepackage{amsfonts}
\usepackage{amsmath}
\usepackage{amssymb}
\usepackage{enumerate}
\usepackage{comment}
\usepackage{hyperref}
\newtheorem{thm}{Theorem}[section]

\newtheorem{cor}[thm]{Corollary}
\newtheorem{lem}[thm]{Lemma}
\newtheorem{prop}[thm]{Proposition}
\newtheorem*{prob*}{Problem}
\newtheorem*{thm*}{Theorem}
\newtheorem*{quest*}{Question}

\theoremstyle{definition}
\newtheorem{defn}[thm]{Definition}
\newtheorem{example}[thm]{Example}

\newtheorem*{defn*}{Definition}
\newtheorem{rem}[thm]{Remark}

\newtheorem{rem*}[thm]{Remark}
\numberwithin{equation}{section}

\newcommand{\mP}{\mathbb {P}}
	\title{$\bigoplus_{p\in P}\mathbb{F}_p$- Systems as Abramov systems}
\date{\today}
\author{Or Shalom}

\begin{document}

	\begin{abstract}
		Let $\mathcal{P}$ be an (unbounded) countable multiset of primes, let $G=\bigoplus_{p\in P}\mathbb{F}_p$. We study the $k$'th universal characteristic factors of an ergodic probability system $(X,\mathcal{B},\mu)$ with respect to some measure preserving action of $G$. We find conditions under which every extension of these factors is generated by phase polynomials  (Theorem \ref{Mainr:thm}) and we give an example of an ergodic $G$-system that is not Abramov. In particular we generalize the main results of Bergelson Tao and Ziegler \cite{Berg& tao & ziegler} who proved a similar theorem in the special case $P=\{p,p,p,...\}$ for some fixed prime $p$. In a subsequent paper \cite{Sh} we use this result to prove a general structure theorem for ergodic $\bigoplus_{p\in P}\mathbb{F}_p$-systems.
	\end{abstract}
	\maketitle
\section{Introduction}
This paper is concerned with the study of the structure of the universal characteristic factors of some ergodic measure preserving systems associated with the action of some abelian groups (formal definitions below). Host and Kra \cite{HK} and independently Ziegler \cite{Z} studied these factors for ergodic $\mathbb{Z}$-actions.
Later, Bergelson Tao and Ziegler \cite{Berg& tao & ziegler} proved a counterpart for $\mathbb{F}_p^\omega$-actions.
The goal of this work is to generalize the results in \cite{Berg& tao & ziegler}. More concretely, we study the structure of the universal characteristic factors associated with the group $\bigoplus_{p\in P} \mathbb{F}_p$ for some countable multiset of primes $P$ (i.e. every prime may appear multiple times).\\
We prove three main results\footnote{All definitions for the terms used here are given later in the paper where these results are formalized.}. The first
is an if and only if condition under which every abelian extension of finite type of the universal characteristic factors of an ergodic $\bigoplus_{p\in P} \mathbb{F}_p$-system $X$ is Abramov (Theorem \ref{Mainr:thm}). More formally, we say that a system is "strongly Abramov" if every abelian extension of finite type is an Abramov system (See Definition \ref{stronglyAbr}). We show that the universal characteristic factors of $X$, $Z_{<1}(X),Z_{<2}(X),...$ are strongly Abramov if and only if every $(\bigoplus_{p\in P} \mathbb{F}_p,Z_{<l}(X),S^1)$-cocycle of finite type in $H^1(\bigoplus_{p\in P}\mathbb{F}_p,Z_{<l}(X),S^1)$ has a representative which is measurable with respect to a "totally disconnected" factor (see Definition \ref{TD:def}).  
This generalizes the main results of Bergelson Tao and Ziegler (\cite[Theorem 3.3, Theorem 8.1]{Berg& tao & ziegler}) who proved a similar theorem in the special case $P=\{p,p,p,...\}$ for some fixed prime $p$.\\ A priori, it is not clear whether all systems are strongly Abramov. Our second result is an example which justifies this definition. In the case where the multiset $P$ is unbounded, we construct an ergodic system $(X,\bigoplus_{p\in P} \mathbb{F}_p)$ which is Abramov but is not strongly Abramov (in particular there exists an extension that is not Abramov). Note that in the case where $P$ is bounded, Bergelson Tao and Ziegler \cite{Berg& tao & ziegler} proved that every ergodic system is strongly Abramov. An interesting fact is that the underline space $X$ in the example above is a solenoid (a finite dimensional compact abelian group that is not a Lie group), these groups are known for their pathological properties. Roughly speaking, this example indicates that these special solenoids form the "building blocks" of every ergodic $\bigoplus_{p\in P} \mathbb{F}_p$ system. In a subsequent paper \cite{Sh} we formalize this statement and prove this to be true. From this we then deduce a complete structure theorem for ergodic $\bigoplus_{p\in P} \mathbb{F}_p$ systems.\\ Our last result states, roughly speaking, that all "non-pathological" ergodic $\bigoplus_{p\in P} \mathbb{F}_p$-systems satisfy the conditions in the first result, and therefore are strongly Abramov. In particular, the last result explains why the use of solenoids as an example for a system that is not Abramov is necessary. \\

\subsection{Universal Characteristic Factors}
We begin with some standard definitions, essentially taken from \cite{Berg& tao & ziegler}.
\begin{defn} A $G$-system is a quadruple $(X,\mathcal{B},\mu,G)$ where $(X,\mathcal{B},\mu)$ is a probability measure space which is separable modulo null sets\footnote{For technical reasons we assume that the space $(X,\mathcal{B},\mu)$ is regular, meaning that $X$ is compact, $\mathcal{B}$ is the Borel $\sigma$-algebra and $\mu$ is a Borel measure.}, together with an action of $G$ on $X$ by measure preserving transformations $T_g:X\rightarrow X$. For every topological group\footnote{All topological groups in this paper are implicitly assumed to be metrizable} $(U,\cdot)$ measurable map $\phi:X\rightarrow U$ and element $g\in G$, we define the shift $T_g\phi = \phi \circ T_g$ and the multiplicative derivative $\Delta_g \phi = \frac{T_g \phi}{\phi}$. We say that a $G$-system $X$ is ergodic if the only functions in $L^2(X)$ which are invariant under the $G$-action are the constant functions.
\end{defn} 
The Gowers-Host-Kra seminorms play an important role throughout this paper

\begin{defn}
	[Gowers Host Kra seminorms for an arbitrary countable abelian group $G$]
	Let $G$ be a countable abelian group, let $X=(X,\mathcal{B},\mu)$ be a $G$-system, let $\phi\in L^\infty (X)$, and let $k\geq 1$ be an integer. The Gowers-Host-Kra seminorm $\|\phi\|_{U^k}$ of order $k$ of $\phi$ is defined recursively by the formula
	\[
	\|\phi\|_{U^1}:=\lim_{N\rightarrow\infty}\frac{1}{|\Phi_N^1|}\|\sum_{g\in\Phi_N^1}\phi\circ T_g\|_{L^2}
	\]
	for $k=1$, and
	\[
	\|\phi\|_{U^k}:=\lim_{N\rightarrow\infty}\left(\frac{1}{|\Phi_N^k|}\sum_{g\in\Phi_N^k}\|\Delta_g\phi\|_{U^{k-1}}\right)^{1/2^k}
	\]
	for $k\geq 1$, where $\phi_N^1,...,\phi_N^k$ are arbitrary $F{\o}lner$ sequences.
\end{defn}
These norms in the special case where $G=\mathbb{Z}/N\mathbb{Z}$ were first introduced by Gowers in \cite{G}, where he derived quantitative bounds for Szemer{\'e}di's Theorem about the existence of arbitrary large arithmetic progressions in sets of positive upper Banach density, \cite{Sz}. Later, in \cite{HK} Host and Kra introduced the above ergodic theoretical version of the Gowers norms (for $G=\mathbb{Z}$).\\ Gowers' work raised a natural question about the structure of functions with large norm (also known as the inverse problem for the Gowers norms). This question was answered partially by Gowers in his work above. Inspired by the work of Host and Kra \cite{HK}, Green and Tao solved this problem for the case $k=3$ in \cite{GT1} and together with Ziegler for general $k$ in \cite{GTZ}. Their work hints about a link between the ergodic theoretical structure of the universal characteristic factors and the inverse problem for the Gowers norms. Surprisingly, if one considers these problems in the context of vector spaces over finite fields this link becomes more concrete. Namely, in \cite{TZ} Tao and Ziegler deduced an inverse theorem for the Gowers norm over finite fields from an ergodic theoretical structure theorem for  $\mathbb{F}_p^\omega$-systems which they established together with Bergelson in \cite{Berg& tao & ziegler}. It is not yet known how to deduce a solution to the inverse problem in the context of $\mathbb{Z}/N\mathbb{Z}$ from the work of Host and Kra.\\
 
 Another approach for this inverse problem is the study of nilspaces. In \cite{SzCam} Szegedy and Camerana introduced a purely combinatorial counter-part of the universal characteristic factors called a nilspace. The idea was to give a more abstract and general notion which describes the "cubic structure" of an ergodic system (See Host and Kra \cite[Section 2]{HK}). In \cite{Sz3} Szegedy and Candela used this notion in order to give an abstract and purely combinatorial answer to the inverse problem for the Gowers norms. Their work answers this question not only for all abelian groups but also for all nilpotent groups. The notion "nilspace" is more general and abstract than the measure-theoretical counterpart. Thus, describing these nilspaces in a concrete way is often a difficult problem on it's own. In a series of papers, \cite{Gut1},\cite{Gut2},\cite{Gut3} Gutman, Manners, and Varj{\'u} studied further the structure of nilspaces (More results in this direction can be found in \cite{Sz1},\cite{Sz2} by Szegedy and Candela). Their work gives a more concrete, less abstract characterization of nilspaces (with some additional properties) in the case where the underline group is finitely generated. By imposing another measure-theoretical aspect to these nilspaces, Gutman and Lian \cite{Gut4} gave an alternative proof of Host and Kra's theorem.

In our work we do not pursue this approach, instead our goal is to generalize the ergodic theoretical structure for other groups. In this paper we're mostly concerned with proving a more general version of the main results of Bergelson Tao and Ziegler \cite[Theorem 4.5]{Berg& tao & ziegler}, in a subsequent paper \cite{Sh} we provide a full structure theorem associated with the group $\bigoplus_{p\in P}\mathbb{F}_p$ which also provides a counter-part to the theorem of Host and Kra \cite{HK} about $\mathbb{Z}$-actions.\\

The Gowers-Host-Kra seminorms correspond to a "factor" of $X$ (See Proposition \ref{UCF} below). We define
\begin{defn} [Factors] \label{Factor:def}
	Let $(X,\mathcal{B}_X, \mu_X, (T_g)_{g\in G})$ be a $G$-system. We say that a $G$-system $(Y,\mathcal{B}_Y,\mu_Y, (S_g)_{g\in G})$ is a factor of $(X,\mathcal{B}_X, \mu_X, (T_g)_{g\in G})$ if there's a measure preserving factor map $\pi_Y^X : X\rightarrow Y$ such that the push-forward of $\mu_X$ by $\pi^X_Y$ is $\mu_Y$ and $\mu^X_Y\circ T_g =S_g \circ \pi^X_Y$ $\mu_X$-a.e. for all $g\in G$.\\
	For a measure-space $U$ and a measurable map $f:Y\rightarrow U$ we define the pull back $(\pi_Y^X )^\star f: X\rightarrow U$ by $(\pi_Y^X )^\star f = f\circ\pi^X_Y$. Similarly, for $f\in L^2(X)$ we write $(\pi_Y^X)_\star f\in L^2(Y)$ for the push forward of $f$. With this notations one has $E(f|Y):=(\pi^X_Y)^\star (\pi^X_Y)_\star f $ where $E(f|Y)$ is the conditional expectation of $f$ to $Y$.\\
	We say that a function $f$ is $\mathcal{B}_Y$-measurable or measurable with respect to $Y$ if $f=E(f|Y)$, or equivalently if $f=(\pi^X_Y)^\star F$ for some $F\in L^2(Y)$.\\
	In this case we refer to $X$ as an extension of the $G$-system $Y$. Finally, we say that a factor $Y$ is generated by a collection $\mathcal{F}$ of measurable functions $f:X\rightarrow \mathbb{C}$ if $Y$ is the minimal factor of $X$ such that all $f\in\mathcal{F}$ are measurable with respect to $Y$.
\end{defn}
We have the following  Proposition\slash Definition (See Host and Kra, \cite[Lemma 4.3]{HK}):
\begin{prop}[Existence and uniqueness of the universal characteristic factors] \label{UCF} Let $G$ be a countable abelian group, let $X$ be a $G$-system, and let $k\geq 1$. Then there exists a factor $Z_{<k}(X)=(Z_{<k}(X),\mathcal{B}_{Z_{<k}(X)},\mu_{Z_{<k}(X)},\pi^X_{Z_{<k}(X)})$ of $X$ with the property that for every $f\in L^\infty (X)$, $\|f\|_{U^{k}(X)}=0$ if and only if $(\pi^{X}_{Z_{<k}(X)})_\star f = 0$ (equivalently $E(f|Z_{<k}(X))=0$). This factor is unique up to isomorphism and is called the $k$'th universal characteristic factor of $X$.
\end{prop}
The structure of the universal characteristic factors for multiple averages for $\mathbb{Z}$-systems was studied independently by Host and Kra in \cite{HK} as a tool in the study of some non-conventional ergodic averages. Those averages were originally introduced by Furstenberg \cite{F1} in his proof of Szemer{\'e}di's Theorem. Ziegler defined these factors differently \cite{Z}. In \cite{Leib} Leibman proved the equivalence.
\begin{thm} [Structure theorem for $Z_{<k}(X)$ for ergodic $\mathbb{Z}$-systems] \cite[Theorem 10.1]{HK} \cite[Theorem 1.7]{Z} \label{HK}
	For an ergodic system $X$, $Z_{<k}(X)$ is an inverse limit of $k$-step nilsystems\footnote{A $k$-step nilsystem is quadruple $(\mathcal{G}/\Gamma,\mathcal{B},\mu,\mathbb{Z})$ where $\mathcal{G}$ is a $k$-step nilpotent Lie group, $\Gamma$ a co-compact subgroup. $\mathcal{B}$ is the Borel $\sigma$-algebra, $\mu$ the induced Haar measure and the action of $\mathbb{Z}$ is given by a left translation by an element in $\mathcal{G}$.}.
\end{thm}
This theorem leads to various multiple recurrence and convergence results in ergodic theory, see for instance \cite{BTZ},\cite{CL84},\cite{CL87},\cite{CL88},\cite{F&W}.
\subsection{Abelian cohomology and some notations}
We use the same notations as in \cite{Berg& tao & ziegler}.
\begin{defn} [Abelian cohomology] Let $G$ be a countable discrete abelian group. Let $X=(X,\mathcal{B},\mu,G)$ be a $G$-system, and let $U=(U,\cdot)$ be a compact abelian group. 
	
	\begin{itemize}
		\item{We let $\mathcal{M}(X,U)$ denote the set of all measurable functions $\phi:X\rightarrow U$, with two functions $\phi,\phi'$ identified if they agree $\mu$-almost everywhere. $\mathcal{M}(X,U)$ is an abelian group under point-wise multiplication, and is a topological group with respect to the topology of convergence in measure.}
		\item {Similarly, let $\mathcal{M}(G,X,U)$ denote the set of all measurable functions $\rho:G\times X\rightarrow U$ with $\rho,\rho'$ being identified if $\rho(g,x)=\rho'(g,x)$ for $\mu$-almost every $x\in X$ and every $g\in G$.}
		\item {We let $Z^{1}(G,X,U)$ denote the subgroup of $\mathcal{M}(G,X,U)$ consisting of those $\rho:G\times X\rightarrow U$ which satisfies that $\rho(g+g',x)=\rho(g,x)\rho(g',T_gx)$ for all $g,g'\in G$ and $\mu$-almost every $x\in X$. We refer to the elements of $Z^{1}(G,X,U)$ as cocycles.}
		\item {Given a cocycle $\rho:G\times X\rightarrow U$ we define the abelian extension $X\times_{\rho} U$ of $X$ by $\rho$ to be the product space $(X\times U,\mathcal{B}_X\times\mathcal{B}_U,\mu_X\times\mu_U)$ where $\mathcal{B}_U$ is the Borel $\sigma$-algebra on $U$ and $\mu_U$ the Haar measure. We define the action of $G$ on this product space by $(x,u)\mapsto (T_gx,\rho(g,x)u)$ for every $g\in G$. In this setting we define the vertical translations  $V_u(x,t)=(x,ut)$ on $X\times_{\rho} U$ for every $u\in U$, we note that this action of $U$ commutes with the $G$-action on this system.}
		\item {If $F\in \mathcal{M}(X,U)$, we define the derivative $\Delta F\in \mathcal{M}(G,X,U)$ of $F$ to be the function $\Delta F(g,x):=\Delta_g F(x)$. We write $B^1(G,X,U)$ for the image of $\mathcal{M}(X,U)$ under the derivative operation. We refer the elements of $B^1(G,X,U)$ as $(G,X,U)$-coboundaries.}
		\item {We say that $\rho,\rho'\in \mathcal{M}(G,X,U)$ are $(G,X,U)$-cohomologous if $\rho/\rho'\in B^1(G,X,U)$.}
	\end{itemize}
\end{defn}
\begin{rem} \label{coh:rem} Observe that if $\rho$ and $\tilde{\rho}$ are $(G,X,U)$-cohomologous, then $X\times_{\rho} U$ and $X\times_{\tilde{\rho}}U$ are measure-equivalent systems. The isomorphism is given by $\pi(x,u)=(x,F(x)u)$ where $F:X\rightarrow U$ is a function such that $\rho = \tilde{\rho}\cdot \Delta F$.
\end{rem}
\subsection{Type of functions}
We introduce the notion of Cubic systems from \cite{HK} (Generalized for arbitrary countable abelian group).
\begin{defn} [Cubic measure spaces] \cite[Section 3]{HK}. Let $X=(X,\mathcal{B},\mu,G)$ be a $G$-system for some countable abelian group $G$. For each $k\geq 0$ we define $X^{[k]} =(X^{[k]},\mathcal{B}^{[k]},\mu^{[k]},G^{[k]})$ where $X^{[k]}$ is the Cartesian product of $2^k$ copies of $X$, endowed with the product $\sigma$-algebra $\mathcal{B}^{[k]}=\mathcal{B}^{2^k}$, $G^{[k]}=G^{2^k}$ acting on $X^{[k]}$ in the obvious manner. We define the cubic measures $\mu^{[k]}$ and $\sigma$-algebras $\mathcal{I}_k\subseteq \mathcal{B}^{[k]}$ inductively. $\mathcal{I}_0$ is defined to be the $\sigma$-algebra of invariant sets in $X$, and $\mu^{[0]}:=\mu$. Once $\mu^{[k]}$ and $\mathcal{I}_k$ are defined, we identify $X^{[k+1]}$ with $X^{[k]}\times X^{[k]}$ and define $\mu^{[k+1]}$ by the formula 
	$$\int f_1(x)f_2(y) d\mu^{[k+1]}(x,y) = \int E(f_1|\mathcal{I}_k)(x)E(f_2|\mathcal{I}_k)(x) d\mu^{[k]}(x)$$
	For $f_1,f_2$ functions on $X^{[k]}$ and $E(\cdot|\mathcal{I}_k)$ the conditional expectation, and $\mathcal{I}_{k+1}$ being the $\sigma$-algebra of invariant sets in $X^{[k+1]}$.

\end{defn}
\begin{defn} [Functions of type $<k$] \cite[Definition 4.1]{Berg& tao & ziegler}  \label{type:def} Let $G$ be a countable abelian group, let $X=(X,\mathcal{B},\mu,G)$ be a $G$-system. Let $k\geq 0$ and let $X^{[k]}$ be the cubic system associated with $X$ and $G$ acting on $X^{[k]}$ diagonally.
	\begin{itemize}
		\item{For each measurable $f:X\rightarrow U$, we define a measurable map $d^{[k]}f:X^{[k]}\rightarrow U$ to be the function $$d^{[k]}f((x_w)_{w\in \{-1,1\}^k}):=\prod_{w\in \{-1,1\}^k}f(x_w)^{\text{sgn}(w)}$$ where $\text{sgn}(w)=w_1\cdot w_2\cdot...\cdot w_k$}
		\item {Similarly for each measurable $\rho:G\times X\rightarrow U$ we define a measurable map $d^{[k]}\rho:G\times X^{[k]}\rightarrow U$ to be the function 
			$$d^{[k]}\rho(g,(x_w)_{w\in \{-1,1\}^k}):=\prod_{w\in \{-1,1\}^k} \rho(g,x_w)^{\text{sgn}(w)}$$}
		\item {A function $\rho:G\times X\rightarrow U$ is said to be a function of type $<k$ if $d^{[k]}\rho$ is a $(G,X^{[k]},U)$-coboundary. We let $\mathcal{M}_{<k}(G,X,U)$ denote the subspace of functions $\rho:G\times X\rightarrow U$ of type $<k$. We let $\mathcal{C}_{<k}(G,X,U)$ denote the subspace of $\mathcal{M}_{<k}(G,X,U)$ which are also cocycles.}
	\end{itemize}
\end{defn}

\begin{defn} [Phase polynomials]
	Let $G$ be a countable abelian discrete group, $X$ be a $G$-system, let $\phi\in L^\infty(X)$, and let $k\geq 0$ be an integer. A function $\phi:X\rightarrow \mathbb{C}$ is said to be a phase polynomial of degree $<k$ if we have $\Delta_{h_1}...\Delta_{h_k}\phi = 1$ $\mu_X$-almost everywhere for all $h_1,...,h_k\in G$. (In particular by setting $h_1=...=h_k=0$ we see that $\phi$ must take values in $S^1$, $\mu_X$-a.e.). We write $P_{<k}(X)=P_{<k}(X,S^1)$ for the set of all phase polynomials of degree $<k$. Similarly, a function $\rho:G\times X\rightarrow \mathbb{C}$ is said to be a phase polynomial of degree $<k$ if $\rho(g,\cdot)\in P_{<k}(X,S^1)$ for every $g\in G$. We let $P_{<k}(G,X,S^1)$ denote the set of all phase polynomials $\rho:G\times X\rightarrow \mathbb{C}$ of degree $<k$.
\end{defn}
We write $\text{Abr}_{<k}(X)$\footnote{It was Abramov who studied systems of this type for $\mathbb{Z}$-actions, see \cite{A}.} for the factor of $X$ generated by $P_{<k}(X)$, and say that $X$ is an Abramov system of order $<k$ if it is generated by $P_{<k}(X)$, or equivalently if $P_{<k}(X)$ spans $L^2(X)$. 
\begin{rem}
	The notion of phase polynomials can be generalized for an arbitrary abelian group $(U,\cdot)$. A function $\phi:X\rightarrow U$ is said to be a phase polynomial of degree $<k$ if $\Delta_{h_1}...\Delta_{h_k}\phi = 1$ $\mu_X$-a.e. for all $h_1,...,h_k\in G$. We let $P_{<k}(X,U)$ and similarly $P_{<k}(G,X,U)$ denote the phase polynomials of degree $<k$ takes values in $U$.
\end{rem}

We recall some properties about functions of type $<k$ and about phase polynomials from \cite{Berg& tao & ziegler}.
We recall some basic facts about functions of type $<k$ from \cite{Berg& tao & ziegler}
\begin{lem} \label{PP}
	Let $G$ be a countable abelian group, let $X=(X,\mathcal{B},\mu,G)$ be an ergodic $G$-system, let $U=(U,\cdot)$ be a compact abelian group, and let $k\geq 0$.
	
	\begin{enumerate} [(i)]
		\item {Every function $f:G\times X\rightarrow U$ of type $<k$ is also of type $<k+1$.}
		\item {The set $M_{<k}(G,X,U)$ is a subgroup of $M(G,X,U)$ and it contains the group $B^1(G,X,U)$ of coboundaries. In particular every function that is $(G,X,U)$-cohomologous to a function of type $<k$, is a function of type $<k$.}
		\item {A function $f:G\times X\rightarrow U$ is a phase polynomial of degree $<k$ if and only if $d^{[k]}f=1$, $\mu^{[k]}$-almost everywhere. In particular every phase polynomial of degree $<k$ is of type $<k$}
		\item{If $f:G\times X\rightarrow U$ is a $(G,X,U)$-coboundary, then $d^{[k]}f:G\times X^{[k]}\rightarrow U$ is a $(G,X^{[k]},U)$-coboundary.}
	\end{enumerate}
	
\end{lem}

\begin{lem} 
	Let $G$ be a countable abelian group, $X$ be a $G$-system and $k\geq 0$.
	\begin{enumerate}[(i)]
		\item{(monotonicity) We have $P_{<k}(X,U)\subseteq P_{<k+1}(X,U)$. In particular $Abr_{<k}(X)\leq Abr_{<k+1}(X)$.}
		\item {(Homomorphism) $P_{<k}(X,U)$ is a group under point-wise multiplication, and for each $h\in H$ $\Delta_h$ is a homomorphism from $P_{<k+1}(X,U)$ to $P_{<k}(X,U)$.}
		\item {(Polynomiality criterion) Conversely, if $\phi:X\rightarrow U$ is measurable and for every $g\in G$, $\Delta_g\phi\in P_{<k}(X,U)$ then $\phi\in P_{<k+1}(X,U)$.}
		\item {(Functoriality) If $Y$ is a factor of $X$ then the pullback $(\pi^X_Y)^\star$ is a homomorphism from $P_{<k}(Y,U)$ to $P_{<k}(X,U)$. Conversely if $f:Y\rightarrow U$ is such that $(\pi^X_Y)^\star f\in P_{<k}(X,U)$ then $f\in P_{<k}(Y,U)$.}
		
	\end{enumerate}
\end{lem}
Bergelson Tao and Ziegler, \cite{Berg& tao & ziegler} proved
\begin{thm}  [Structure theorem for $Z_{<k}(X)$ for ergodic $\mathbb{F}_p^\omega$-systems] \label{BTZ} There exists a constant $C(k)$ such that for any ergodic $\mathbb{F}_p^\omega$-system $X$, $L^2(Z_{<k}(X))$ is generated by phase polynomials of degree $<C(k)$.  Moreover if $p$ is sufficiently large with respect to $k$ then $C(k)=k$.
\end{thm}
\subsection{Main results}
We say that a system $X$ is of order $<k$ if $X=Z_{<k}(X)$; We begin with the following result of Host and Kra \cite{HK} generalized for arbitrary discrete countable abelian group action.
\begin{prop}  [Order $<k+1$ systems are abelian extensions of order $<k$ systems]\label{abelext:prop} \cite[Proposition 6.3]{HK} Let $G$ be a discrete countable abelian group, let $k\geq 1$ and $X$ be an ergodic $G$-system of order $<k+1$. Then $X$ is an abelian extension $X=Z_{<k}(X)\times_{\rho} U$ for some compact abelian metrizable group $U$ and a cocycle $\rho:G\times Z_{<k}(X)\rightarrow U$ of type $<k$.
\end{prop}
In other words, for a countable discrete abelian group $G$, an ergodic $G$-system of order $<k+1$ is isomorphic to a tower of abelian extensions \begin{equation} \label{structure}U_0\times_{\rho_1}U_1\times...\times_{\rho_k}U_k
\end{equation} such that for each $1\leq i\leq k$, $\rho_i:G\times Z_{<i-1}(X)\rightarrow U_i$ is a cocycle of type $<i$. We call $U_1,...,U_k$ the structure groups of $X$. We are particularly interested in the structure of these groups. We define,
	\begin{defn} [Totally disconnected systems and Weyl systems] \label{TD:def} Let $G$ be a discrete countable abelian group. Let $X$ be an ergodic system of order $<k$ as above
		\begin{itemize}
			\item {We say that $X$ is a totally disconnected system if $U_0,U_1,...,U_{k-1}$ are totally disconnected groups.}
			\item {We say that $X$ is a Weyl-system if for every $1\leq i \leq k-1$ the cocycle $\rho_i$ is a phase polynomial of degree $<O_k(1)$.}
		\end{itemize} Note that we will show that totally disconnected systems are isomorphic to Weyl-systems (See Theorem \ref{TDisweyl}).
	\end{defn}
	
	We are particularly interested in systems whose abelian extensions of finite type are Abramov. Formally,
	
\begin{defn} \label{stronglyAbr}
    Let $X$ be an ergodic $G$-system. We say that $X$ is \emph{strongly Abramov} if for every $m\in\mathbb{N}$ there exists $l_m\in\mathbb{N}$ such that for any compact abelian group $U$ and a cocycle $\rho:G\times X\rightarrow U$ of type $<m$ the extension $X\times_\rho U$ is Abramov of order $<l_m$. (In particular $X$ is Abramov)
\end{defn}

Our first results is an if and only if criterion for this property. Namely,
\begin{thm} [A criterion for strongly Abramov] \label{Mainr:thm}
	Let $P$ be a countable (unbounded) multiset of primes and let $G=\bigoplus_{p\in P}\mathbb{F}_p$. Let $X=Z_{<k}(X)$ be an ergodic $G$-system of order $<k$. Then the factors $Z_{<1}(X),Z_{<2}(X),...,Z_{<k}(X)$ are strongly Abramov if and only if there exists a totally disconnected factor $Y$ of $X$ such that  the homomorphism induced by the factor maps $\pi_l:Z_{<l}(X)\rightarrow Z_{<l}(Y)$
	$$\pi_l^\star:H^1_{<m}(G,Z_{<l}(Y),S^1)\rightarrow H^1_{<m}(G,Z_{<l}(X),S^1)$$ is onto for every $m\in\mathbb{N}$. Moreover, we can take $l_m=O_{k,m}(1)$ and if $k,m\leq \min P$ we can take $l_m=m+1$.
\end{thm}
In particular, with a little work it follows that
\begin{cor} \label{cor} Let $G$ be as in Theorem \ref{Mainr:thm} and let $X$ be an ergodic totally disconnected $G$-system of order $<k$. Then for an integer $m\in\mathbb{N}$ any cocycle $\rho:G\times X\rightarrow S^1$ of type $<m$ is $(G,X,S^1)$-cohomologous to a phase polynomial of degree $<O_{k,m}(1)$. If $\min P >k,m$ then it is cohomologous to a phase polynomial of degree $<m$.
\end{cor}

A priori, it is not clear that systems which are not strongly Abramov exists. In fact, Bergelson Tao and Ziegler \cite{Berg& tao & ziegler} proved that in the case where $P$ is bounded every system is strongly Abramov. On the other hand, in the case of $\mathbb{Z}$-actions, Furstenberg and Weiss showed that not all systems are strongly Abramov. They constructed an extension of the $2$-dimensional torus by a $1$-dimensional torus that is not Abramov of any order (In \cite{HK02} Host and Kra worked out their example in details). Our second result is a counter-part of their example in the case of $\bigoplus_{p\in P}\mathbb{F}_p$ -systems. We show that when $P$ is unbounded there exists an ergodic Abramov system which is not strongly Abramov. Namely, 
\begin{thm}  [An Abramov system that is not strongly Abramov]\label{example:thm}
	Let $P$ be the set of prime numbers. There exists a solenoid\footnote{A solenoid is a compact abelian finite dimensional group that is not a Lie group. These are known for their pathological properties.} $U$ with an ergodic $\bigoplus_{p\in P}\mathbb{F}_p$ action such that $(U,\bigoplus_{p\in P}\mathbb{F}_p)$ is of order $<2$ and a cocycle $\rho:\bigoplus_{p\in P}\mathbb{F}_p\times U\rightarrow S^1$ of type $<2$ that is not $(\bigoplus_{p\in P}\mathbb{F}_p,U,S^1)$-cohomologous to a phase polynomial of any degree. The extension $U\times_\rho S^1$ is an ergodic $\bigoplus_{p\in P}\mathbb{F}_p$-system of order $<3$ that is not Abramov of any order.
\end{thm}
In Furstenberg and Weiss' example, the group $U$ is the $2$ dimensional torus. While a torus can be given a structure of an ergodic $\bigoplus_{p\in P}\mathbb{F}_p$-system it is impossible to generalize their example with $U$ being a torus. In fact in Theorem \ref{Main:thm} below we show that the torus and a much larger family of "non-pathological" ergodic systems are strongly Abramov. Before we formulate the theorem we recall the following results of Host and Kra \cite{HK} for $\mathbb{Z}$-actions and the result of Bergelson Tao and Ziegler \cite{Berg& tao & ziegler} for $\mathbb{F}_p^\omega$-actions.\\

A Toral system $X$ of order $<k$ is an ergodic system whose structure groups $U_1,...,U_{k-1}$ such that $U_1$ is a Lie group and for all $2\leq i\leq k-1$, $U_i$ is isomorphic to a finite dimensional torus.
\begin{thm}[\cite{HK}] Let $X$ be an ergodic $\mathbb{Z}$-system of order $<k$. Then $X$ is an inverse limit of Toral systems of order $<k$.
\end{thm}
We say that a group $(U,\cdot)$ is $n$-torsion if $u^n=1_U$ for all $u\in U$.
\begin{thm}[\cite{Berg& tao & ziegler}]
	Let $X$ be an ergodic $\mathbb{F}_p^\omega$-system of order $<k$. Then the structure groups $U_0,U_1,...,U_{k-1}$ are $p^m$-torsion for some $m=O_k(1)$. In particular $X$ is a totally disconnected system.
\end{thm}
We can see now that Theorem \ref{Mainr:thm} and in particular Corollary \ref{cor} generalize Bergelson Tao and Ziegler's structure theorem (Theorem \ref{BTZ}).\\
We show that systems of the form "Toral by totally disconnected" are satisfying the condition in Theorem \ref{Mainr:thm}. Formally, 
\begin{defn} [Splitting condition] \label{split:def}
	Let $G$ be a countable discrete abelian group and let $X$ be an ergodic $G$-system of order $<k$. Write $X=U_0\times_{\rho_1} U_1\times...\times_{\rho_{k-1}}U_{k-1}$ as in (\ref{structure}). Suppose that for every $0\leq i\leq k-1$ the group $U_i$ is isomorphic to $T_i\times D_i$ where $T_i$ is a torus (possibly zero dimensional, or infinite dimensional) and $D_i$ is totally disconnected. If the projection of each $\rho_i$ to $D_i$ is invariant under $T_1\times...\times T_{i-1}$ we say that $X$ satisfies the splitting condition (or that $X$ splits in short). In this case we denote by $\mathcal{T}(X):=\prod_{i=0}^{k-1}T_i$ the torus part of $X$ and by $D(X):=\prod_{i=0}^{k-1} D_i$ the totally disconnected part of $X$. 
\end{defn} 
\begin{rem}
There exists an ergodic system $(\mathbb{T},\bigoplus_{p\in P}\mathbb{F}_p)$ where $\mathbb{T}$ is a torus and a cocycle $\rho:\bigoplus_{p\in P}\mathbb{F}_p\times\mathbb{T}\rightarrow \Delta$ (which is non-constant) into a disconnected group $\Delta$ such that the extension $X:=\mathbb{T}\times_\rho \Delta$ is isomorphic to a solenoid. In other words, the splitting condition identifies all the "non-pathological" ergodic $\bigoplus_{p\in P}\mathbb{F}_p$ systems.
\end{rem}
Our main result implies that these systems are strongly Abramov,
\begin{thm} [Main Theorem]\label{Main:thm} 
	Let $P$ be a countable (unbounded) multiset of primes and let $G=\bigoplus_{p\in P}\mathbb{F}_p$. Let $X$ be a system of order $<k$ which splits and let $m\in\mathbb{N}$. Then any cocycle $\rho:G\times X\rightarrow S^1$ of type $<m$ is $(G,X,S^1)$-cohomologous to a phase polynomial of degree $<O_{k,m}(1)$.
\end{thm}
We also prove an "high characteristic" version,
\begin{thm}\label{MainH:thm}
	In the settings above. If $k,m<\min P $ then $\rho$ is $(G,X,S^1)$-cohomologous to a phase polynomial of degree $<m$.
\end{thm}
Theorem \ref{Main:thm} is the main difficulty in this paper and it implies Theorem \ref{Mainr:thm} (see Section \ref{proof}). In order to demonstrate some of the difficulties we give an example of an ergodic $\bigoplus_{p\in\mathbb{F}_p}$ action on the torus.
\begin{example}  [The torus as an ergodic $\bigoplus_{p\in P} \mathbb{F}_p$-system]\label{Example}
	Consider the circle $S^1$ with the Borel $\sigma$-algebra and the Haar measure. We define an action of $\bigoplus_{p\in P} \mathbb{F}_p$ by a homomorphism $\varphi:\bigoplus_{p\in P} \mathbb{F}_p\rightarrow S^1$ which is given by the formula $$\varphi((g_p)_{p\in P}) = \prod_{p\in P} w_p^{g_p}$$
	where $w_p$ is the first root of unity of degree $p$. This formula is well-defined because $g_p=0$ for all but finitely many $p\in P$ (including multiplicities).\\
	Let $f\in L^\infty(X)$ be an invariant function, by comparing the Fourier coefficients of $f$ and $f\circ T_g$ one can show that if $P$ is unbounded, then $f$ must be a constant. In other words $X:=(S^1,\bigoplus_{p\in P}\mathbb{F}_p)$ is ergodic. Moreover, it is easy to see that the characters $\chi_n(z)=z^n$ form an orthonormal basis of eigenfunctions in $L^2(X)$. It follows that $X$ is a system of order $<2$. Moreover, every phase polynomial of degree $<2$ is a constant multiple of a character.
	\end{example}
	We notice the following fact,\\
	\textbf{Claim:} Let $X$ be as in the example. Then, every phase polynomial $F:X\rightarrow S^1$ is of degree $<2$. (i.e. there are no phase polynomials of higher degree)
	\begin{proof} Let $F:X\rightarrow S^1$ be a phase polynomial of degree $<m$ and consider the map $s\mapsto \Delta_s F$ from $S^1$ to $P_{<m}(X)$. By Corollary \ref{ker:cor} this map takes values in $P_{<1}(X)$, hence by ergodicity $\Delta_s F = \chi(s)$ for some constant $\chi(s)$ depending only on $s$. Using the cocycle identity (i.e. $\Delta_{st} F(x) = \Delta_s F(tx) \Delta_t F(x)$) we conclude that $\chi(st)=\chi(s)\chi(t)$. It follows that $\Delta_s (F(x)/\chi(x))=1$ for every $s\in S^1$. Hence $F/\chi$ is a constant and so $F$ is a polynomial of degree $<2$.
	\end{proof}
	From this example we see that the many results of Bergelson Tao and Ziegler \cite{Berg& tao & ziegler} for $\mathbb{F}_p^\omega$-systems can not be generalized for $\bigoplus_{p\in P}\mathbb{F}_p$ where $P$ is unbounded. More concretely we see that
	
	\begin{enumerate}
		\item {Ergodic $\bigoplus_{p\in P} \mathbb{F}_p$-systems of finite order need not be totally disconnected.}
		\item {Phase polynomials need not take finitely many values.}
		\item {From the last claim we see that some phase polynomials, like the identity map need not have phase polynomial roots.}
	\end{enumerate}
	
	Fortunately, these properties hold in the case where the system is totally disconnected (See Definition \ref{TD:def} above) and we manage to reduce the matters to this case.\\
Below is an overview of the structure of the paper
\subsection{Overview of the structure of the paper}
Most of the paper will be devoted to the proof of Theorem \ref{Main:thm} (Main theorem) which is also the main component in the proof of Theorem \ref{Mainr:thm} (Criterion for strongly Abramov).\\ 
The paper begins by following the arguments of Host and Kra \cite{HK} and Bergelson Tao and Ziegler \cite{Berg& tao & ziegler} in order to reduce the original problem into solving a certain Conze-Lesgine type equation. More concretely, by the theory of Host and Kra (see Proposition \ref{abelext:prop}) every system of order $<k$ takes the form  $$X=U_0\times_{\rho_1}\times U_1\times_{\rho_2}U_2\times...\times_{\rho_{k-1}}U_{k-1}$$ for some compact abelian groups $U_0,...,U_{k-1}$, and cocycles $\rho_i:Z_{<i-1}(X)\rightarrow U_i$. Using a proof by induction we may assume that $\rho$ satisfies a Conze-Lesigne type equation with respect to the automorphisms on $X$. Namely, there exists $m$ (bounded by $k$) such that \begin{equation} \label{eq:overview}\Delta_t \rho = p_t\cdot \Delta F_t
\end{equation} for all automorphism $t:X\rightarrow X$ where $p_t\in P_{<m}(\bigoplus_{p\in P}\mathbb{F}_p,X,S^1)$ and $F_t\in \mathcal{M}(X,S^1)$. \\

The next fundamental step is to reduce matters to the case where the underline system $X$ is totally disconnected (i.e. the groups $U_0,U_1,...,U_{k-1}$ are totally disconnected).\\ 
Inductively we construct factors $$X=X_k\rightarrow X_{k-1}\rightarrow...\rightarrow X_0$$ with the following properties
\begin{enumerate}
	\item {For $0\leq l<k$, $X_l$ is obtained from $X_{l+1}$ by quotienting out the torus of $U_l$.}
	\item {For $0\leq l<k$, the torus of $U_{l-1}$ acts on $X_l$ by automorphisms.}
	\item {Every cocycle $\rho:G\times X_l\rightarrow S^1$ of type $<k$ is $(G,X_l,S^1)$-cohomologous to a cocycle that is measurable with respect to $X_{l-1}$.}
\end{enumerate}
The last property is the most difficult to prove and the main work is taken care of in Sections \ref{fdred:sec} and \ref{tdred:sec}. In section \ref{fdred:sec} we work in full generality (we don't assume that $X$ splits). We use the totally disconnected nature of $\hat G$, the dual of $G:=\bigoplus_{p\in P} \mathbb{F}_p$, together with a "linearization trick" of Furstenberg and Weiss \cite{F&W}. We show that $\rho:G\times X_l\rightarrow S^1$ is cohomologous to a cocycle which is invariant under some subgroup $J_l$ of $U_l$ such that $U_l/J_l$ is a finite dimensional group (See Definition \ref{FD:def}). In the case where $X$ splits, this means that $U_l/J_l$ is a product of a finite torus and a totally disconnected group. We note that working in this generality leads to even stronger version of Theorem \ref{Main:thm} (See the discussion in section \ref{fdred:sec}). However, for the sake of simplicity we choose to stick with the version above. In section \ref{tdred:sec} we finish the argument by another application of the linearization lemma, however this time combined with an idea of Ziegler from \cite{Z}.\\
The proof of the totally disconnected case follows the general idea of Bergelson Tao and Ziegler \cite{Berg& tao & ziegler}, however the multiplicity of generators of different prime orders in $G$ leads to some new difficulties. In section \ref{TSTweyl:sec} we show that if $X$ splits then the totally disconnected part of $X$ is in fact a direct product of finite groups, moreover we show that these finite groups are of the form $C_{p^d}$ where $p\in P$ and $d=O_k(1)$ (and $d=1$ if $\min P$ is sufficiently large, see Theorem \ref{TST:thm}). This is a counterpart of \cite[Theorem 4.8]{Berg& tao & ziegler}, in which Bergelson Tao and Ziegler proved that every ergodic $\mathbb{F}_p^\omega$-system of order $<k$ is totally disconnected and of $p^m$ torsion for some $m=O_k(1)$ (and $m=1$ if $p$ sufficiently large with respect to $k$). Finally, in section \ref{low:sec} (and in section \ref{high:sec} for high characteristic) we finish the proof following the ideas of Bergelson Tao and Ziegler, by working on each of the primes separately. The high characteristic case of Theorem \ref{Mainr:thm} follows almost the same arguments as in \cite{Berg& tao & ziegler} and therefore we omit most of it.

\textbf{Acknowledgment} I would like to thank my adviser Prof. Tamar Ziegler for many valuable discussions and suggestions.

\section{Standard reductions} \label{sr:sec}
Throughout the rest of the paper it will be convenient to use the letter $G$ to denote the group $\bigoplus_{p\in P}\mathbb{F}_p$ for some fixed and possibly unbounded countable multiset of primes $P$.\\
Our first goal in this section is to prove the following direction of Theorem \ref{Mainr:thm} assuming Theorem \ref{Main:thm} (Compare with \cite[Theorem 3.8]{Berg& tao & ziegler}). We show,
\begin{thm*} Let $X$ be an ergodic system of order $<k$. Let $1\leq l \leq k$ and suppose that $Z_{<l}(X)$ is Abramov of order $<O_k(1)$ and that there exists a totally disconnected factor $Y$ and a factor map $\pi_l:Z_{<l}(X)\rightarrow Z_{<l}(Y)$ such that $$\pi_l^\star: H^1_{<m} (G,Z_{<l}(Y),S^1)\rightarrow H^1_{<m}(G,Z_{<l}(X),S^1)$$ is onto for every $m\in\mathbb{N}$. Then, for every $1\leq l\leq k$, $Z_{<l}(X)$ is strongly Abramov.
\end{thm*}
Note that $Z_{<l}(X)$ is an extension of $Z_{<l-1}(X)$ and therefore by proving this direction of Theorem \ref{Mainr:thm} by induction on $l$ we may assume that $Z_{<l}(X)$ is Abramov.
\begin{proof}
We prove the claim by induction on $l$; If $l=1$ then $Z_{<1}(X)$ is a point and the claim is trivial. Let $1< l \leq k$, and suppose that the claim has already been proven for smaller values of $l$. Let $Y$ be as in the theorem; For $m\in\mathbb{N}$ let $\rho:G\times Z_{<l}(X)\rightarrow U$ be a cocycle of type $<m$ for some compact abelian group $U$. From the assumption, we see that for all $\chi\in\hat U$, $\chi\circ\rho$ is $(G,Z_{<l}(X),S^1)$-cohomologous to a cocycle $\rho_\chi$ which is measurable with respect to $Z_{<l}(Y)$. Let $\pi:Z_{<l}(X)\rightarrow Z_{<l}(Y)$ be the projection map.  By \cite[Corollary 7.8]{HK} (See Lemma \ref{Cdec:lem} below) $\pi_\star\rho_\chi$ is a cocycle of type $<m$ on $Z_{<l}(Y)$. Since $Z_{<l}(Y)$ is totally disconnected, Theorem \ref{Mainr:thm} implies that $\pi_\star\rho_\chi$ is $(G,Z_{<l}(Y),S^1)$-cohomologous to a phase polynomial of degree $<l_m$ for some $l_m=O_{k,m}(1)$. Lifting everything back  to $Z_{<l}(X)$ using $\pi^\star$ we conclude that $\chi\circ\rho$ is $(G,Z_{<l}(X),S^1)$-cohomologous to a phase polynomial of degree $<l_m=O_{k,m}(1)$. Write $\chi\circ \rho = q_\chi \cdot \Delta F_\chi$ for some $F_\chi:Z_{<l}(X)\rightarrow S^1$ and a phase polynomial $q_\chi: G\times Z_{<l}(X)\rightarrow S^1$. We conclude that the map $\Phi(x,u) = \chi(u)\cdot \overline{F}(x)$ is a phase polynomial of degree $<l_m+1$ on $Z_{<l}(X)\times_\rho U$. Since $Z_{<l}(X)$ is Abramov of degree $<r_k$ for some $r_k=O_k(1)$ we conclude that $\overline{F}$ can be approximated by phase polynomials of some bounded degree. We conclude that $\chi$ can be approximated by phase polynomials of degree $<\max\{r_k,l_m+1\} = O_{k,m}(1)$. Since $L^2(Z_{<l}(X)\times_\rho U)$ is generated by $Z_{<l}(X)$ and these characters this completes the proof.
\end{proof}
This proves one of the directions of Theorem \ref{Mainr:thm}. The other direction is proved in the end of this paper after we prove Theorem \ref{Main:thm}. We now turn to the proof of Theorem \ref{Main:thm}\\

We need a definition,
\begin{defn} [Automorphism] \label{Aut:def} Let $X$ be a $G$-system. A measure-preserving transformation $u:X\rightarrow X$ is called an Automorphism if the induced action on $L^2(X)$ by $V_u(f)=f\circ u$ commutes with the action of $G$. In this case we define the multiplicative derivative with respect to $u$ by $\Delta_u f = V_uf\cdot \overline{f}$
\end{defn}
Automorphisms arise naturally from Host-Kra's theory. For instance the group $U$ acts by automorphisms on the abelian extension $Y\times_\rho U$ by $V_u (y,t)=(y,tu)$.\\
Our goal in this section is to show that Theorem \ref{Main:thm} follows from the following two theorems; The first asserts that in order to prove Theorem \ref{Main:thm} it suffices to do so on a totally disconnected system. Namely,
\begin{thm} [Reduction to a totally disconnected factor]\label{TDred:thm}
	Let $k\geq 1$, let $X$ be an ergodic $G$-system of order $<k$ which splits. Then there exists a totally disconnected ergodic $G$-system $Y$, with a factor map $\pi:X\rightarrow Y$ such that for every $m\in\mathbb{N}$ and any cocycle $\rho:G\times X\rightarrow S^1$ of type $<m$ for which
	\begin{equation} \label{Conze-Lesigne}\Delta_t \rho = p_t\cdot \Delta F_t\end{equation} for all automorphisms $t:X\rightarrow X$, where $p_t\in P_{<m}(G,X,S^1)$ and  $F_t:X\rightarrow S^1$ is a measurable map, we have that $\rho$ is $(G,X,S^1)$-cohomologous to $\pi^\star \rho'$ for a cocycle $\rho':G\times  Y\rightarrow S^1$.
\end{thm}
The second theorem is a totally disconnected version of Theorem \ref{Main:thm}. As in \cite{Berg& tao & ziegler} the theorem holds for general functions of type $<m$ that may not be cocycles.

\begin{thm} [The totally disconnected case]\label{MainT:thm} Let $k,m\geq 0$ and let $X$ be an ergodic totally disconnected $G$-system of order $<k$. Let $f:G\times X\rightarrow S^1$ be a function of type $<m$. Then $f$ is $(G,X,S^1)$-cohomologous to a phase polynomial of degree $<O_{k,m}(1)$.
\end{thm}
In retrospect, the reduction to a totally disconnected case is a reasonable approach assuming that Theorem \ref{Mainr:thm} holds.
\subsection{Proof of Theorem \ref{Main:thm} assuming Theorem \ref{TDred:thm} and Theorem \ref{MainT:thm}}
We recall some results from Bergelson Tao and Ziegler \cite{Berg& tao & ziegler}
\begin{lem}  [Decent of type for cocycles] \cite[Proposition 8.11]{Berg& tao & ziegler}  \label{Cdec:lem}
	Let $k\geq 0$ and let $X$ be an ergodic $G$-system of order $<k$. Let $Y$ be a factor of $X$, with factor map $\pi:X\rightarrow Y$. Suppose that $\rho:G\times Y\rightarrow S^1$ is a cocycle. If $\pi^\star \rho$ is of type $<k$, then $\rho$ is of type $<k$.
\end{lem}
\begin{lem}  [Differentiation Lemma] \cite[Lemma 5.3]{Berg& tao & ziegler}\label{dif:lem} Let $k\geq 1$, and let $X$ be a $G$-system of order $<k$. Let $f:G\times X\rightarrow S^1$ be a function of type $<m$ for some $m\geq 1$. Then for every automorphism $t:X\rightarrow X$ which preserves $Z_{<k}(X)$ the function $\Delta_t f(x) := f(tx)\cdot\overline{f}(x)$ is of type $<m-\min(m,k)$.
\end{lem}

We now prove Theorem \ref{Main:thm} assuming Theorem \ref{TDred:thm} and Theorem \ref{MainT:thm}. The proof is similar to the proof of \cite[Theorem 5.4]{Berg& tao & ziegler}.
\begin{proof} [Proof of Theorem \ref{Main:thm} assuming Theorem \ref{TDred:thm} and Theorem \ref{MainT:thm}]
	Let $k,m,X,\rho$ be as in Theorem \ref{Main:thm}. We claim that for every $0\leq j \leq m$ and automorphisms $t_1,...,t_j:X\rightarrow X$ we have,
	
	$$\Delta_{t_1}...\Delta_{t_j}\rho \in P_{<O_{k,m,j}(1)}(G,X,S^1)\cdot B^1(G,X,S^1)$$
	
	We establish this by downward induction on $m$.
	For $j=m$ the claim follows by iterating Lemma \ref{dif:lem} (with no polynomial term). Assuming inductively that the claim is true for $j+1$. Let $Y$ be as in Theorem \ref{TDred:thm} and let $\pi:X\rightarrow Y$ be the projection map. We have that $\Delta_{t_1}...\Delta_{t_j} \rho $ is $(G,X,S^1)$-cohomologous to $\pi^\star \rho_{t_1,...,t_j}$ where $\rho_{t_1,...,t_j}:G\times Y\rightarrow S^1$ is a cocycle. By assumption $\rho$ is of type $<m$, and $t_{j+1}$ commutes with the $G$-action we therefore have that $\Delta_{t_1}...\Delta_{t_j}\rho$ is of type $<m$ and so by Lemma \ref{PP} (ii), $\pi^\star \rho_{t_1,...,t_j}$ is also of type $<m$. Thus, Lemma \ref{Cdec:lem} implies that $\rho_{t_1,...,t_j}:G\times Y\rightarrow S^1$ is a cocycle of type $<m$. Since $\rho_{t_1,...,t_j}$ is defined on a totally disconnected system $Y$ we see from Theorem \ref{MainT:thm} that $\rho_{t_1,...,t_j}$ is $(G,Y,S^1)$-cohomologous to a phase polynomial $P_{t_1,...,t_j}:G\times Y\rightarrow S^1$ of degree $<O_{k,m,j}(1)$. Lifting everything up by $\pi$, we conclude that $\rho_{t_1,...,t_j}$ is $(G,X,S^1)$-cohomologous to $\pi^\star P_{t_1,...,t_j}$ which by functoriality is a phase polynomial of degree $<O_{k,m,j}(1)$. This completes the proof of the claim, Theorem \ref{Main:thm} now follows from the case $j=0$.
\end{proof}
The rest of the paper is devoted to the proof of Theorem \ref{TDred:thm} and Theorem \ref{MainT:thm} assuming the induction hypothesis of Theorem \ref{Main:thm}.
\section{Reduction to a finite dimensional $U$}\label{fdred:sec}
In the next couple of sections we develop a theory which eventually allow us to reduce the proof of Theorem \ref{Main:thm} to the case where $X$ is totally disconnected. In this section we develop tools that later will be used to replace each of the structure groups with a finite dimensional quotient (See Definition \ref{FD:def}, below). We work in general settings (we don't assume that $X$ splits) and so we potentially deal with some "pathological" groups. We believe that working in this generality gives a better understanding of the example in Theorem \ref{example:thm}. Moreover, one can adapt our proof to prove a more general version of Theorem \ref{Main:thm}. More concretely, it is possible to replace the torus in the splitting condition with any compact connected abelian group $U$ that has no non-trivial local isomorphism to $\hat G$ (i.e. there is no open neighborhood $U'$ of the identity in $U$ and a non-trivial map $\varphi:U'\rightarrow \hat G$ such that $\varphi(u\cdot u')=\varphi(u)\cdot \varphi(u')$ whenever $u,u',u\cdot u'\in U'$). For the sake of simplicity we do not prove Theorem \ref{Main:thm} in this generality. \begin{defn} [Definition and properties of finite dimensional compact abelian groups] \label{FD:def} \cite[Theorem 8.22]{HM} The following conditions are equivalent for a compact abelian group $H$ and a natural number $n$:
	
	\begin{enumerate}
		\item {There is a compact zero dimensional subgroup $\Delta$ of $H$ and an exact sequence $$1\rightarrow \Delta \rightarrow H\rightarrow \mathbb{T}^n\rightarrow 1$$ where $\mathbb{T}^n$ is the $n$-dimensional torus.}
		\item {There is a compact zero dimensional subgroup $\Delta$ of $H$ and a continuous quotient homomorphism $\varphi:\Delta\times\mathbb{R}^n\rightarrow H$ which has a discrete kernel.}
	\end{enumerate}
	We say that a compact abelian group $H$ is finite dimensional if it satisfies at least one of these conditions.
	
\end{defn}
\begin{rem}
	Note that condition $(1)$ in the previous definition is equivalent to the existence of some zero dimensional subgroup $D$ of $H$ and an exact sequence $$1\rightarrow D\rightarrow H\rightarrow \mathbb{T}^n\times C_k\rightarrow 1$$ for some finite group $C_k$. For if $H/D\cong \mathbb{T}^n\times C_k$ then $C_k$ is a subgroup of $H/D$ hence the isomorphism theorem implies that there exists $D\subseteq D'\subseteq H$ such that $D'/D\cong C_k$. It follows that $D'$ is also zero dimensional group and $H/D'\cong \mathbb{T}^n$ as required.
\end{rem}

\begin{defn} [free action] \cite[Definition B.2]{Berg& tao & ziegler}
	Let $U$ be a locally compact group acting on a probability space $X$ in a measure-preserving way. The action of $U$ is said to be free if $X$ is measure-equivalent to a system of the form $Y\times U$ and the action of $u\in U$ is given by $V_u(y,v)=(y,uv)$. 
\end{defn}
In this section we study cocycles $\rho:G\times X\rightarrow S^1$ which satisfy a Conze-Lesgine type equation (See \cite{CL84}, \cite{CL87}, \cite{CL88} for relevant work of Conze and Lesigne about $\mathbb{Z}$-systems of order $<3$) with respect to some group $U$ of measure-preserving transformations on $X$. That is, we assume that there exists some $m\in\mathbb{N}$ such that for every $u\in U$, $\Delta_u \rho = p_u\cdot \Delta F_u$ for some phase polynomial $p_u:G\times X\rightarrow S^1$ and a measurable map $F_u:X\rightarrow S^1$ (In retrospective, if Theorem \ref{Main:thm} holds then any cocycle of type $<m$ satisfies this equation with respect to any group of automorphisms). The main result in this section is the following,
\begin{prop}\label{FD:prop}
	Let $X$ be an ergodic $G$-system. Let $U$ be a compact abelian group acting freely on $X$ and commuting with the action of $G$. Let $m\geq 0$ and $\rho:G\times X\rightarrow S^1$ be a cocycle so that for every $u\in U$ we have that $\Delta_u\rho = p_u\cdot \Delta F_u$ for some phase polynomial $p_u\in P_{<m}(G,X,S^1)$ and a measurable map $F_u:X\rightarrow S^1$. Then $\rho$ is $(G,X,S^1)$-cohomologous to a cocycle which is invariant under some closed connected subgroup $J$ of $U$ for which $U/J$ is a finite dimensional compact abelian group.
\end{prop}

We will take advantage of the following results of Bergelson Tao and Ziegler \cite{Berg& tao & ziegler}.

\begin{lem}  [Separation Lemma] \cite[Lemma C.1]{Berg& tao & ziegler}\label{sep:lem} Let $X$ be an ergodic $G$-system, let $k\geq 1$, and let $\phi,\psi\in P_{<k}(X,S^1)$ be such that $\phi/\psi$ is non-constant. Then $\|\phi-\psi\|_{L^2(X)} \geq \sqrt{2}/2^{k-2}$.
\end{lem}
Since $X$ is compact, $L^2(X)$ is separable. From this it follows that up to a constant multiple, there are only countably many phase polynomials of a given degree.
\begin{lem}  [Measure selection Lemma] \cite[Lemma C.4]{Berg& tao & ziegler}\label{sel:lem} Let $X$ be an ergodic $G$-system, and let $k\geq 1$. Let $U$ be a compact abelian group. If $u\mapsto h_u$ is Borel measurable map from $U$ to $\mathcal{P}_{<k}(G,X,S^1)\cdot \mathcal{B}^1(G,X,S^1) \subseteq \mathcal{M}(G,X,S^1)$ where $\mathcal{M}(G,X,S^1)$ is the group of measurable maps of the form $G\times X\rightarrow S^1$ endowed with the topology of convergence in measure, then there is a Borel measurable choice of $f_u,\psi_u$ (as functions from $U$ to $M(X,S^1)$ and $U$ to $P_{<k}(G,X,S^1)$ respectively) obeying that $h_u = \psi_u \cdot \Delta f_u$.
\end{lem}

\begin{lem} [Straightening nearly translation-invariant cocycles] \cite[Lemma C.9]{HK} \label{cob:lem}
	Let $X$ be an ergodic $G$-system, let $K$ be a compact abelian group acting freely on $X$ and commuting with the $G$-action, and let $\rho:G\times X\rightarrow S^1$ be such that $\Delta_k\rho$ is a $(G,X,S^1)$-coboundary for every $k\in K$, then $\rho$ is $(G,X,S^1)$-cohomologous to a function which is invariant under the action of some open subgroup of $K$. 
\end{lem}
\begin{rem}
	Note that if $K$ is connected then it has no non-trivial open subgroups (see Lemma \ref{connectedcomponent}). In this case we have that $\rho$ is $(G,X,S^1)$-cohomologous to a function which is invariant under $K$. Moreover it is important to note that such result does not work for cocycles which takes values in an arbitrary compact abelian group.
\end{rem}
\begin{lem} [Linearization of the $p_u$-term] \label{lin:lem}
	Let $X$ be an ergodic $G$-system, let $U$ be a compact abelian group acting freely on $X$ and commuting with the action of $G$. Let $\rho:G\times X\rightarrow S^1$ be a cocycle and suppose that there exists $m\in\mathbb{N}$ such that for every $u\in U$ there exist phase polynomials $p_u\in P_{<m}(G,X,S^1)$ and a measurable map $F_u:X\rightarrow S^1$ such that $\Delta_u\rho = p_u\cdot \Delta F_u$. Then there exists a measurable choice $u\mapsto p'_u$ and $u\mapsto F'_u$ such that $\Delta_u\rho = p'_u\cdot \Delta F'_u$ for phase polynomials $p'_u\in P_{<m}(G,X,S^1)$ which satisfies that $p'_{uv} = p'_u\cdot V_u p'_v$ whenever $u,v,uv\in U'$ where $U'$ is some neighborhood of $U$.
\end{lem}
The proof of this Lemma is given in \cite{Berg& tao & ziegler} as part of the proof of Proposition 6.1.
\begin{rem}
	The idea of linearizing the term $p_u$ as in the lemma above was originally introduced by Furstenberg and Weiss in \cite{F&W}. Later, this idea was used by Ziegler in \cite{Z} and by Bergelson Tao and Ziegler in \cite{Berg& tao & ziegler} who studied the structure of the universal characteristic factors for the groups $\mathbb{Z}$ and $\mathbb{F}_p^\omega$ respectively. This idea also plays a role, somehow behind the scenes in the paper of Host and Kra \cite{HK}.
\end{rem}

Finally, we prove the following Lemma about phase polynomial cocycles 

\begin{lem}\label{ker:lem}
	Let $X$ be an ergodic $G$-system, let $H$ be a compact abelian group acting freely on $X$ which commutes with the $G$-action and let $p:H\rightarrow P_{<m}(G,X,S^1)$ be a measurable map whose image are phase polynomials of degree $<m$ that are also cocycles. Suppose that $p$ satisfies $p(hk)=p(h)V_h p(k)$ for all $h,k\in H$. Then $p$ is trivial on the connected component $H_0$ of the identity in $H$.
\end{lem}
\begin{proof}
	We apply Corollary \ref{ker:cor} once for every $g\in G$. From Lemma \ref{connectedcomponent} we see that the restriction of $p$ to the connected component $H_0$ is a homomorphism which takes values in $P_{<1}(G,X,S^1)\cap Z^1(G,X,S^1)$.
	By ergodicity we can identify $P_{<1}(G,X,S^1)\cap Z^1(G,X,S^1)$ with the dual group $\hat G$ of $G$ (equipped with the usual topology of point-wise convergence). We conclude that $p|_{H_0}:H_0\rightarrow \hat G$ is a measurable homomorphism from a connected group to a totally disconnected group, hence trivial (see Lemma \ref{AC:lem}).
\end{proof}
We can finally prove Proposition \ref{FD:prop}.
\begin{proof}[Proof of Proposition \ref{FD:prop}]
	Let $X,U,\rho,m,p_u,F_u$ be as in Proposition \ref{FD:prop}. By Lemma \ref{sel:lem} we may assume that $u\mapsto p_u$ and $u\mapsto F_u$ are measurable. Moreover, by Lemma \ref{lin:lem} we may assume that there exists an open neighborhood $U'$ of the identity in $U$ on which $p_u$ is a cocycle in $u$ (i.e. $p_{uv}=p_u V_u p_v$ whenever $u,v,uv\in U'$). It is well known that every compact abelian group can be approximated by Lie groups (Theorem \ref{approxLieGroups}). In other words, there exists a closed subgroup $J'$ contained in $U'$ such that $U/J'$ is a Lie group. Let $p:J'\rightarrow P_{<m}(G,X,S^1)$ be the map $j\mapsto p_j$. From Lemma \ref{ker:lem} it follows that $p$ is trivial on the connected component of the identity in $J'$, let $J:=J'_0$. We conclude that $\Delta_j \rho$ is a coboundary for all $j\in J$. By Lemma \ref{cob:lem} and Lemma \ref{connectedcomponent} we have that $\rho$ is $(G,X,S^1)$-cohomologous to a cocycle which is invariant under the action of $J$.\\
	
	It is left to show that $U/J$ is a finite dimensional group. We have the following exact sequence $$1\rightarrow J'/J\rightarrow U/J\rightarrow U/J'\rightarrow 1$$
Here $J'/J$ is totally disconnected and $U/J'$ is a Lie group. By the structure theorem for compact abelian Lie groups (Theorem \ref{structureLieGroups}) we know that $U/J'$ is a finite extension of a Torus. Hence, $U/J'$ is a finite dimensional compact abelian group as in property (1) of Definition \ref{FD:def}.
\end{proof}
In the next section we work with these finite dimensional groups. Our results in the next section will only be useful in the case where the system $X$ splits.
\section{Reduction to the totally disconnected group case} \label{tdred:sec}
In this section we work with a finite dimensional group $U$ acting on a system $X$. By the structure of finite dimensional groups (Definition \ref{FD:def}) we can write $U\cong\left(\mathbb{R}^n\times\Delta\right) /\Gamma$ for some $n\in\mathbb{N}$, a totally disconnected compact abelian group $\Delta$ and a discrete subgroup $\Gamma$ of $\mathbb{R}^n\times \Delta$. Throughout this section it will be convenient to identify $\mathbb{R}^n\times\{1\}$ with $\mathbb{R}^n$ and $\{1\}\times \Delta$ with $\Delta$.
\begin{defn}
	Let $(H,+)$ and $(K,\cdot)$ be two abelian groups. We say that a map $\varphi:H\rightarrow K$ is a homomorphism on some $A\subseteq H$ if $\varphi(x+y)=\varphi(x)\cdot\varphi(y)$ whenever $x,y\in A$.
\end{defn}
Below is a simple lemma about homomorphisms on open subsets of $\mathbb{R}^n$.
\begin{lem} [Lifting of homomorphisms on open sets of $\mathbb{R}^n$]\label{lift:lem} Let $H$ be a compact abelian group and let $\varphi:\mathbb{R}^n\rightarrow H$ be a measurable map.
	Suppose that there exists an open ball around zero $W\subseteq \mathbb{R}^n$ such that $\varphi(v+w)=\varphi(v)+\varphi(w)$ whenever $v,w\in W$. Then there exists a homomorphism $\tilde{\varphi}:\mathbb{R}^n\rightarrow H$ such that $\tilde{\varphi}$ and $\varphi$ agree on $W$.
\end{lem}
\begin{rem}
	Lemma \ref{lift:lem} is a special case of \cite[Lemma 4.7]{Z}. Also we note that if $\varphi$ is measurable then so is $\tilde{\varphi}$.
\end{rem}

\begin{proof}
	Write $W=B(0,r)$ for the open ball of radius $r$ around $0$ for some $r>0$. First we prove that we can extend $\varphi$ to a function $\varphi_1$ which is a homomorphism on $1.5W = B(0,1.5 r)$.\\
	We construct a map $\varphi_1 : \mathbb{R}^n\rightarrow H$ as follows: 
	If $a\not\in 3W$ we let $\varphi_1(a)=1$ (or any other element in $H$). Otherwise, choose any $x\in W$ and $y\in 2W$ with $a=x+y$ and let $\varphi_1(a)=\varphi(x)\varphi(y)$. Assume for now that $\varphi_1$ is well defined (i.e. independent of the choice of $x,y$), we claim that $\varphi_1$ extends $\varphi$ and is a homomorphism on $1.5W$. Indeed, let $x,y\in 1.5W$ then write $x=1.5u$ and $y=1.5v$ for $u,v\in W$. We have, $$\varphi_1(1.5u+1.5v)=\varphi_1 (u+(0.5u+0.5v)+v)$$
	as $u,v\in W$ we see that $0.5u+0.5v\in W$ and so $0.5u+1.5v\in 2W$. From the assumption that $\varphi_1$ is well defined we have
	$$\varphi_1 (u+(0.5u+0.5v)+v)=\varphi(u)\varphi(0.5u+0.5v+v)=\varphi(u)\varphi(0.5u+0.5v)\varphi(v)$$
	where the last equality follows from the linearity of $\varphi$ on $W$. Using the linearity of $\varphi$ few more times we have
	$$\varphi(u)\varphi(0.5u+0.5v)\varphi(v)=\varphi(u)\varphi(0.5u)\varphi(0.5v)\varphi(v)=\varphi_1(1.5u)\varphi_1(1.5v)$$
	Combining everything we conclude that $\varphi_1(x+y)=\varphi_1(x)\varphi_1(y)$ as desired.\\ 
	
	It is left to show that $\varphi_1$ is well defined. Let $x,x'\in W$ and $y,y'\in 2W$ be so that $x+y=x'+y'$ we want to show that $\varphi(x)\varphi(y)=\varphi(x')\varphi(y')$\\
	By assumption, $x-x'=x+(-x')$ is a sum of two elements in $W$ hence we have that $\varphi(x-x')=\varphi(x)\varphi(x')^{-1}$ and so $\varphi(x)=\varphi(x')\varphi(x-x')$. It is thus enough to prove that $\varphi(y')=\varphi(y)\varphi(x-x')$. We have 
	$$\varphi(\frac{y'}{2})=\varphi(\frac{x-x'}{2}+\frac{y}{2})=\varphi(\frac{x-x'}{2})\varphi(\frac{y}{2}) $$
	Now, as $\frac{y'}{2},\frac{(x-x')}{2}$ and $\frac{y}{2}$ are elements in $W$ we see that $$\varphi(y')=\varphi(\frac{y'}{2})^2=(\varphi(\frac{x-x'}{2}))^2(\varphi(\frac{y}{2}))^2=\varphi(x-x')\varphi(y)$$
	We conclude that $\varphi_1$ is well defined.  
	\\
	
	Replace $\varphi$ with $\varphi_1$ and $W$ with $1.5W$, by the same arguments as before we can construct $\varphi_2$ which extends $\varphi_1$ on $1.5W$ and is a homomorphism on $(1.5)^2W$. Continue this way, there exists a sequence $\varphi_1,\varphi_2,....$ each extends the other, such that $\varphi_n$ is a homomorphism on $(1.5)^nW$ for every $n\in\mathbb{N}$.\\
	
	Thus, for every $x\in \mathbb{R}^n$ there exists $m$ so that $x\in B(0,(1.5)^mr)$, let $\tilde{\varphi}(x)=\varphi_m(x)$. It is easy to see that $\tilde{\varphi}$ is well defined, is a homomorphism and it agrees with $\varphi$ on $W$.
\end{proof}
Now, given a finite dimensional compact abelian group $U$ we consider it's connected component $U_0$ which is also finite dimensional (see \cite[Corollary 8.24]{HM}). We use the lifting Lemma to obtain some results on the map $p:U_0\rightarrow P_{<m}(G,X,S^1)$ given by $u\mapsto p_u$.
\begin{lem}\label{real:lem}
	Let $X$ be an ergodic $G$-system. Let $U$ be a finite dimensional compact abelian group acting freely on $X$ and commuting with the action of $G$, let $\rho:G\times X\rightarrow S^1$ be a cocycle and $m\geq 1$. Suppose that for every $u\in U$ there exists a phase polynomial $p_u\in P_{<m}(G,X,S^1)$ and a measurable map $F_u:X\rightarrow S^1$ such that $\Delta_u\rho = p_u\cdot \Delta F_u$. Let $U_0$ be the connected component of the identity in $U$ and write $U_0=\left(\mathbb{R}^n\times\Delta\right)/\Gamma$ for some $n\in\mathbb{N}$, a totally disconnected group $\Delta$ and a discrete subgroup $\Gamma$. Let $\pi:\mathbb{R}^n\times \Delta \rightarrow U_0$ be the projection map. Then
	\begin{enumerate}
		\item{There exists an open subgroup $V$ of $\Delta$ such that $$p_{\pi(uv)}=p_{\pi(u)}p_{\pi(v)}$$ whenever $u,v\in V$.}
		\item {There exists an open ball $W\subseteq\mathbb{R}^n$ such that $p_{\pi(w)}=1$ for all $w\in W$.}
		\item {For all $w\in\mathbb{R}^n$ we have that $p_{\pi(w)}$ is a $(G,X,S^1)$-coboundary.}
	\end{enumerate}
\end{lem}
\begin{proof}
	In order to keep the same notations as above, we use the multiplicative notation to denote the additive operation of $\mathbb{R}^n$.\\
	
	Applying the linearization Lemma (Lemma \ref{lin:lem}), there exists an open neighborhood $U'$ of the identity in $U_0$ on which $p_u$ is a cocycle in $u$. Now, $\pi^{-1}(U')$ is an open neighborhood of the identity in $\mathbb{R}^n\times \Delta$ and so by Proposition \ref{opensubgroup} it contains a neighborhood $W\times V$ where $W$ is an open ball around zero in $\mathbb{R}^n$ and $V$ is an open subgroup of $\Delta$. We then have that $\pi(V)$ is a subgroup of $U_0$ and it lies inside $U'$ we conclude that $p_{\pi(uv)} = p_{\pi(u)} V_{\pi(u)} p_{\pi(v)} = p_{\pi(u)}p_{\pi(v)}$ where the last equality follows from Proposition \ref{pinv:prop}. This proves $(1)$.\\
	Let $p:U_0\rightarrow P_{<m}(G,X,S^1)$ be the map $u\mapsto p_u$. We lift $p$ by $\pi$, we have that $p\circ \pi : \mathbb{R}^n\times \Delta\rightarrow P_{<m}(G,X,S^1)$ is a homomorphism on $W\times V$. Applying Lemma \ref{lift:lem} once for every $v\in V$ and then gluing together, we see that there exists a homomorphism $\tilde{p}:\mathbb{R}^n\times V\rightarrow P_{<m}(G,X,S^1)$ which extends $p\circ \pi$ on $W\times V$.  We restrict $\tilde{p}$ to $\mathbb{R}^n$. Since the index of $P_{<1}(X,S^1)$ in $P_{<m}(X,S^1)$ is at most countable (Lemma \ref{sep:lem}) and $\mathbb{R}^n$ has no proper measurable subgroups of countable index (Corollary \ref{openker}), we conclude that $\tilde{p}|_{\mathbb{R}^n}$ takes values in $P_{<1}(G,X,S^1)$. Therefore the range of $\tilde{p}|_{\mathbb{R}^n}$ can be identified with $\hat G$. As there is no non-trivial homomorphism from $\mathbb{R}^n$ to $\hat G$, we conclude that $\tilde{p}$ is trivial on $\mathbb{R}^n$. Since $\tilde{p}$ and $p\circ \pi $ agree on $W$ we have that $p_w=1$ for all $w\in \pi(W)$. This proves $(2)$.\\
	For $(3)$ we denote $$H=\{s\in\mathbb{R}^n : p_{\pi(s)} \text{ is a coboundary}\}$$ using the facts that $p_1=1$, $p_{st}$ is cohomologous to $p_s V_s p_t = p_s p_t$ and $p_s$ is cohomologous to $\overline{p_{s^{-1}}}$ for all $s,t\in U_0$ we have that $H$ is a subgroup of $\mathbb{R}^n$. Moreover from $(2)$ we have that $H$ contains the open ball $W$, therefore $H=\mathbb{R}^n$. This proves $(3)$.
\end{proof}

In this situation we would like to use Lemma \ref{cob:lem} in order to eliminate the $\mathbb{R}^n$ part of $U_0$. Unfortunately this is impossible, because the image of $\mathbb{R}^n$ is a dense, not necessarily compact subgroup of $U_0$. In fact the image of $\mathbb{R}^n$ is closed if and only if $U_0$ is a torus. Therefore as a corollary we have 
\begin{prop} \label{con:prop}
	In the settings of Lemma \ref{real:lem} if $U_0$ is a torus then $\rho$ is $(G,X,S^1)$-cohomologous to a cocycle which is invariant under the action of $U_0$
\end{prop}

Assuming that $X$ splits, we combine this result with the results of the previous section. We have,
\begin{thm} [Eliminating the connected components of a system $X$ of order $<k$ which splits] \label{con:thm} Let $k\geq 1$ be such that Theorem \ref{Main:thm} has already been proven for smaller values of $k$. Let $X$ be an ergodic $G$-system of order $<k$ which splits. We can then write $X=U_0\times_{\rho_1} U_1\times...\times_{\rho_{k-1}}U_{k-1}$, where $\rho_i:G\times Z_{<i}(X)\rightarrow U_i$ is a cocycle of type $<i$ for all $1\leq i \leq k-1$.
	Let $\rho:G\times X\rightarrow S^1$ be a cocycle, such that for some $m\geq 0$ and  any automorphism $u:X\rightarrow X$,  $\Delta_u\rho$ is $(G,X,S^1)$-cohomologous to a phase polynomial $p_u\in P_{<m}(G,X,S^1)$. Then $\rho$ is $(G,X,S^1)$-cohomologous to a cocycle $\tilde{\rho}:G\times X\rightarrow S^1$ which is invariant under the action of $U_{j,0}$ for all $0\leq j \leq k-1$.
\end{thm}

\begin{proof}
	Fix $1\leq j \leq k-1$. We prove the theorem by induction on $k$. If $k=1$ then $X$ is just a point and the claim in the theorem is trivial. Now, assume that the claim holds for smaller values of $k$. We use the induction hypothesis in order to replace the cocycles $\rho_{j+1},...,\rho_{k-1}$ with cocycles that are invariant under the action of $U_{j,0}$. The cocycle $\rho_{j+1}:Z_{<j+1}(X)\rightarrow U_{j+1}$ is of type $<j+1$ and so by Theorem \ref{Main:thm} it's projection to the torus of $U_{j+1}$ is $(G,Z_{<j+1}(X),T(U_{j+1}))$-cohomologous to a phase polynomial. Therefore this projection of $\rho_{j+1}$ satisfies the assumptions of this Theorem, hence by the induction hypothesis it is  $(G,Z_{<j+1}(X),T(U_{j+1}))$-cohomologous to a cocycle which is invariant under $U_{j,0}$. On the other hand since $X$ splits the projection of $\rho_{j+1}$ to the totally disconnected part of $U_{j+1}$ is invariant under $U_{j,0}$. Therefore by remark \ref{coh:rem} we can replace $\rho_{j+1}$ with a cocycle that is invariant under $U_{j,0}$. Direct computation reveals that in this case $U_{j,0}$ acts by automorphisms on $Z_{<j+2}(X)$. Inductively, we can repeat the same argument for $\rho_{j+2},...,\rho_{k-1}$. At the end of the day we replace  $X$ with an isomorphic system in which the cocycles $\rho_{j+1},...,\rho_{k-1}$ are invariant under $U_{j,0}$.\\ In this case, the action of $U_{j,0}$ commutes with the $G$-action on $X$.\\ We now use the theory of section \ref{fdred:sec}. From Proposition \ref{FD:prop} it follows that $\rho$ is $(G,X,S^1)$-cohomologous to a cocycle $\rho':G\times X\rightarrow S^1$ which is invariant under some connected subgroup $J$ of $U_{j,0}$ such that $U_{j,0}/J$ is a finite dimensional torus. Since $\rho_{j+1},...,\rho_{k-1}$ are invariant under $U_{j,0}$ they're also invariant under $J$, in particular we can define a factor $Y = U_0\times_{\rho_1}U_1\times...\times_{\rho_j} U_j/J\times_{\rho_{j+1}'} U_{j+1}\times_{\rho_{j+2}'}...\times_{\rho_{k-1}'}U_{k-1}$ of $X$ where $\rho_{j+1}',...,\rho_{k-1}'$ are induced from $\rho_{j+1},...,\rho_{k-1}$ under the projection $U_j\rightarrow U_j/J$. Let $\pi:X\rightarrow Y$ be the factor map $\pi(u_0,...,u_{k-1}) = (u_0,...,u_j+J,...,u_{k-1})$. Since $\rho'$ is invariant under $J$ it is equal to $\pi^\star \rho''$ where $\rho'':G\times Y\rightarrow S^1$ is a cocycle. Now $U_{j,0}/J$ is a finite dimensional torus acting freely on $Y$ and commuting with the $G$-action. We want to apply Proposition \ref{con:prop}, but first we have to show that $\rho''$ satisfies a Conze-Lesigne type equation.\\
	As $\rho'$ is $(G,X,S^1)$-cohomologous to $\rho$ we have that 
	\begin{equation} \label{eq'}
	\Delta_u \rho' = p_u \cdot \Delta F_u
	\end{equation}
	For the same polynomials $p_u\in P_{<m}(G,X,S^1)$. We have that $\rho'$ is invariant under the action of $J$. Since $J$ is connected we conclude by Proposition \ref{pinv:prop} and equation (\ref{eq'}) that $\Delta F_u$ is invariant under the action of $J$. Since the action of $J$ commutes with the $G$-action, ergodicity implies that $\Delta_j F_u = \chi_u(j)$ for all $j\in J$ for some $\chi_u:U_j\rightarrow S^1$ (We can extend $\chi_u$ to a character of $U_j$).
	
	Applying Theorem \ref{Main:thm}, which by assumption is already proven for smaller values of $k$. We conclude that $\chi_u\circ\rho_{j-1}$ is $(G,Z_{<i}(X),S^1)$-cohomologous to a phase polynomial. Let $\phi_{\chi_u}$ be as in the proof of Theorem \ref{Main:thm}. We have that $\phi_{\chi_u}:Z_{<j}(X)\rightarrow S^1$ is a phase polynomial of degree $<O_k(1)$, and it satisfies $\Delta_j \phi_{\chi_u} = \chi_u(j)$ for all $j\in J$. Let $\tilde{\phi}_{\chi_u}=(\pi^X_{Z_{<j}(X)})^\star\phi_{\chi_u}$ we have that $F'_u:=F_u/\tilde{\phi}_{\chi_u}$ is invariant under $J$ and that $\Delta_u \rho' = p'_u\cdot \Delta F'_u$ where $p'_u = p_u\cdot \Delta \tilde{\phi}_{\chi_u}$ is a phase polynomial of degree $<O_{k,m}(1)$. Now $\rho'$, $p'_u$ and $F'_u$ are all invariant under $J$ and so there exists $\tilde{p}_u :G\times Y\rightarrow S^1$ and $\tilde{F}_u:Y\rightarrow S^1$ such that $p_u = \pi^\star \tilde{p}_u$ and $F_u = \pi^\star \tilde{F}_u$. It follows that $\Delta_u\rho'' = \tilde{p}_u \cdot \Delta \tilde{F}_u$ for all $u\in U_j/J$ also by functoriality we have that $\tilde{p}_u$ are also phase polynomials of degree $<O_{k,m}(1)$.
	
	We can therefore apply Proposition \ref{con:prop}. It follows that $\rho''$ is $(G,Y,S^1)$-cohomologous to a cocycle $\tilde{\rho}:G\times Y\rightarrow S^1$ which is invariant under $U_{j,0}/J$. Therefore, $\pi^\star \tilde{\rho}$ is $(G,X,S^1)$-cohomologous to $\rho'=\pi^\star \rho''$, and is invariant under $U_{j,0}$. As $\rho$ is $(G,X,S^1)$-cohomologous to $\rho'$ it is also $(G,X,S^1)$-cohomologous to $\pi^\star \tilde{\rho}$ which completes the proof.
\end{proof}
\begin{rem}
	It is also possible to use a similar argument as in \cite[Lemma B.11]{Berg& tao & ziegler}.
\end{rem}
Theorem \ref{TDred:thm} follows by repeatedly applying the theorem above. Each time eliminating the connected component of the next group. After $(k-1)$-iterations we get a totally disconnected factor of $X$. Our goal now is to study totally disconnected systems.
\section{Topological structure theorem and Weyl systems} \label{TSTweyl:sec}
The main result of this section is the following theorem (Compare with \cite[Lemma 4.7]{Berg& tao & ziegler})
\begin{thm}  [Totally disconnected extension is necessarily a direct product of finite $p$-groups of bounded exponent] \label{TST:thm} Let $k\geq 2$ be an integer such that Theorem \ref{Main:thm} has already been proven for smaller values of $k$. Let $X$ be a totally disconnected system of order $<k$ and suppose that $X=Z_{<k-1}(X)\times_{\sigma} U$ for some totally disconnected group $U$. Then there exists $d=O_k(1)$, constants $m_1,m_2,... \leq d$, and primes $p_1,p_2,...$ such that $U\cong \prod_{n=1}^\infty C_{p_n^{m_n}}$ as topological groups. Moreover for $p_i$ sufficiently large ($p_i>k$) we can take $m_i=1$.
\end{thm}
Phase polynomials on totally disconnected systems takes finitely many values (See Proposition \ref{TDPV:prop}, and Theorem \ref{HTDPV:thm}). Such polynomials have phase polynomial roots (See Corollary \ref{roots}). These results imply the following useful Lemma,
\begin{lem} \label{valp:lem}
	Let $X$ be a totally disconnected ergodic $G$-system of order $<k$. Let $\rho:G\times X\rightarrow S^1$ be a cocycle which takes values in $C_n$, and suppose that $\rho$ is $(G,X,S^1)$-cohomologous to a phase polynomial of degree $<d$, then $\rho$ is $(G,X,C_n)$-cohomologous to a phase polynomial of degree $<O_{d,k}(1)$.
\end{lem}
\begin{proof}
	Write 
	\begin{equation}  \label{valp:eq0}
	\rho=q\cdot \Delta F
	\end{equation} then, 
	\begin{equation} \label{valp:eq1}
	1=\rho^n = q^n\cdot \Delta F^n
	\end{equation}
	It follows that $P:=F^n$ is a phase polynomial of degree $<d+1$.\\
	As $X$ is totally disconnected, Theorem \ref{TDPV:thm} implies that $P$ takes values in some closed subgroup $H$ of $S^1$ where $H\cong C_{{p_1}^{l_1}}\times...\times C_{{p_m}^{l_m}}$ for $l_1,...,l_m=O_k(1)$ and distinct primes $p_1,..,p_m$.\\
	Let $\pi_i: H\rightarrow C_{p_i^{l_i}}$ be one of the coordinate maps and write $P_i = \pi_i\circ P$, clearly $P_i$ is also a phase polynomial of degree $<d+1$ and we have that $P=\prod_{i=1}^m P_i$.
	Therefore, our goal is to find for each $1\leq i \leq m$ a phase polynomial $\psi_i$ of degree $<O_{d,k}(1)$ that is also an $n$'th root of $P_i$. \\
	
	Fix $1\leq i \leq m$, then $P_i$ takes values in $C_{{p_i}^{l_i}}$. Write $n=p_i^{r_i}\cdot n'$ for some integer $n'$ that is co-prime to $p_i$. First, we find an $(n')$'th root of $P_i$. To do this, let $\alpha\in\mathbb{N}$ be such that $n'\cdot \alpha = 1 \text{ mod } p_i^{l_i}$, and let $\phi_i = P_i^{\alpha}$ we conclude that $$\phi_i^{n'} = P_i {n'\alpha} = P_i$$ Clearly $\phi_i$ is also a phase polynomial of degree $<d+1$. It is therefore left to find a $p_i^{l_i}$'th root for $\phi_i$. We have two cases, if $p_i,r_i=O_{k,d}(1)$ then the claim follows from Corollary \ref{roots}. Otherwise, suppose that either $p_i$ or $r_i$ are sufficiently large. Let $G=G_{p_i}\oplus G'$ where $G_{p_i}$ is the subgroup of $G$ of elements of order $p_i$. In this case we claim that $q(g,x)^n=1$ for all $g\in G_{p_i}$. Indeed, if $r_i$ is sufficiently large then Proposition \ref{PPC} implies $q(g,x)^n=1$ for all $g\in G_{p_i}$. Otherwise if $p_i$ is sufficiently large $(p_i>k)$ then by Theorem \ref{HTDPV:thm} we have that $q(g,x)$ takes values in $C_{p_i}$, hence $q(g,x)^n = 1$.\\
From equation (\ref{valp:eq1}) we see that in both cases $P$ is invariant under $g\in G_{p_i}$ and therefore so is $P_i$. Recall that $P_i$ takes values in $C_{p_i^{l_i}}$ and so from Proposition \ref{PPC} it is invariant under $g\in G$ whose order is coprime to $p$ as well. It follows that $P_i$ is $G$-invariant, hence by ergodicity it is a constant. Since $\phi_i$ is some power of $P_i$ it is also a constant. Hence $\phi_i$ has an $p_i^{r_i}$'th root.\\

We conclude that either way, there exists a phase polynomial $\psi_i$ of degree $<O_{d,k}(1)$ with $\psi_i^n = P_i$. Now we glue all coordinates together, let $\psi:X\rightarrow H$ be the product of all coordinates $\psi(x)=\psi_1(x)\cdots\psi_m(x)$,  as $\psi_1,...,\psi_m$ are phase polynomials of degree $<O_{k,d}(1)$, so is $\psi$. Since $\psi_i^n=P_i$ we conclude that $\psi^n=P$.\\

To finish the proof we now let $F'=F/\psi$ and $q'=q\cdot \Delta \psi$. From equation (\ref{valp:eq0}) we see that $$\rho = q' \cdot \Delta F'$$
	As $\psi$ is an $n$'th root of $P=F^n$ we have that $F'$ takes values in $C_n$ and therefore so is $q'$. Moreover as $\psi$ is a phase polynomial of degree $<O_{k,d}(1)$ we have that so is $q'$ as required.
\end{proof} 
\begin{proof} [Proof of Theorem \ref{TST:thm}]
	Let $X=Z_{<k-1}(X)\times_\sigma U$ be as in the theorem. By Proposition \ref{Sylow} we can write $U=\prod_p U_p$ where $U_p$ are the $p$-sylow subgroups of $U$.\\
	Fix any prime $p$ and let $\chi:U_p\rightarrow S^1$ be any continuous character of $U_p$. As $U_p$ is a $p$-group, the image of $\chi$ is a cyclic group $C_{p^n}$ for some $n\in\mathbb{N}$. By Theorem \ref{Main:thm}, $\chi\circ\sigma$ is $(G,X,S^1)$-cohomologous to a phase polynomial of degree $<O_k(1)$. Therefore, by Lemma \ref{valp:lem} it is $(G,X,C_{p^n})$-cohomologous to a phase polynomial $q:G\times X\rightarrow C_{p^n}$ of degree $<O_k(1)$ (potentially higher bound than before). It follows by Lemma \ref{PPC} that $q$ is trivial for all $g$ of order co-prime to $p$. If $g$ is of order $p$ then by Proposition \ref{TDPV:prop} we have that $q^{p^d}(g,x)=1$ and if $p$ is sufficiently large (greater than $k$)  then by Theorem \ref{HTDPV:thm} we can take $d=1$. By the cocycle identity it follows that $q^{p^d}(g,x)=1$ for all $g\in G$. Since $\sigma$ is cohomologous to $q$ we conclude that $\sigma^{p^d}$ is a coboundary.\\
	The system $Z_{<k}(X)\times_{\sigma^{p^d}} C_{p^{n-d}}$ is a factor of $Z_{<k-1}(X)\times_\sigma U$ and therefore a factor of $X$. Since $X$ is ergodic so is every factor, which means that $n=d$.\\
	We conclude that $U_p$ is a $p^d$-torsion subgroup for some $d=O_k(1)$. Theorem \ref{torsion} implies that $U_p$ is a direct product of copies of $C_{p^r}$ for $r\leq d$. As required.
	\end{proof}
Using this result we can finally prove that such systems are Weyl (Assuming the induction hypothesis of Theorem \ref{Main:thm}). Namely,
\begin{thm} [Totally disconnected systems are Weyl] \label{TDisweyl} Let $k\geq 1$ be an integer such that Theorem \ref{Main:thm} has already been proven for all smaller values of $k$. Let $X$ be a totally disconnected system of order $<k$, then $X$ is isomorphic to a Weyl system.
\end{thm}
\begin{proof} [Proof of Theorem \ref{TDisweyl}] We prove by induction on $k$. If $k=1$ then $X$ is a point and the claim is trivial. Let $k\geq 2$ and assuming inductively that the theorem holds for systems of order $k-1$. Write $X=Z_{<k-1}(X)\times_\sigma U$, by induction hypothesis we can assume that $Z_{<k-1}(X)$ is a Weyl system. Moreover, from Theorem \ref{TST:thm} we have that $U\cong \prod_{i=1}^\infty C_{p_i^{m_i}}$ where $m_i=O_k(1)$ and $m_i=1$ if $p_i$ is sufficiently large. Let $\tau_i:U\rightarrow C_{p_i^{m_i}}$ be one of the coordinate maps. We think of $C_{p_i^{m_i}}$ as a subgroup of $S^1$, we get from Theorem \ref{Main:thm} that $\tau_i\circ\sigma$ is $(G,Z_{<k-1}(X),S^1)$-cohomologous to a phase polynomial of degree $<O_k(1)$ (into $S^1$). Finally, from Lemma \ref{valp:lem} we have that $\tau_i\circ\sigma$ is $(G,Z_{<k-1}(X),C_{p_i^{m_i}})$-cohomologous to a phase polynomial $q_i:G\times X\rightarrow C_{p_i^{m_i}}$ of possibly higher but bounded ($<O_k(1)$) degree. Since this is true for every coordinate map, we conclude that $\sigma$ is $(G,Z_{<k-1}(X),U)$-cohomologous to $q:G\times Z_{<k-1}(X)\rightarrow U$ where $q(x):=(q_1(x),q_2(x),...)$. Since all of the phase polynomials $q_i$ are of bounded degree which only depends on $k$ (and is independent of $i$) it is easy to see that $q$ is of bounded degree. Since cohomologous cocycles defines isomorphic systems (see Remark \ref{coh:rem}), it follows that $X=Z_{<k-1}(X)\times_\sigma U$ is isomorphic to $Z_{<k-1}(X)\times_q U$. Since $q$ is a phase polynomial and $Z_{k-1}(U)$ is a Weyl system this completes the proof. 	
\end{proof}

\section{Proof of Theorem \ref{MainT:thm}} \label{low:sec}
The proof of Theorem \ref{MainT:thm} follows by similar methods as in \cite{Berg& tao & ziegler}. However the multiplicity of generators of different prime orders in $G$ leads to some new difficulties in the "finite group case" which can be solved by working out each prime separately.
\subsection{Reduction of Theorem \ref{MainT:thm} to solving a Conze-Lesigne type equation on a totally disconnected system}

Similarly to Theorem \ref{Main:thm}, using the differentiation lemma (Lemma \ref{dif:lem}) we can translate Theorem \ref{MainT:thm} to solving a Conze-Lesigne type equation. It is thus left to show
\begin{thm}  [Conze-Lesigne type equation for functions on totally disconnected system]\label{TDCL:thm}
	Let $k\geq 1$ be such that Theorem \ref{Main:thm} has already been proven for smaller values of $k$ and let $X=Z_{<k-1}(X)\times _{\rho} U$ be a totally disconnected ergodic $G$-system of order $<k$. Let $f:G\times X\rightarrow S^1$ be a function of type $<m$ for some $m\in\mathbb{N}$ and suppose that
	
	$$\Delta_t f \in P_{<m}(G,X,S^1)\cdot B^1(G,X,S^1)$$
	for all automorphisms $t:X\rightarrow X$. Then $f$ is $(G,X,S^1)$-cohomologous to $P\pi^\star \tilde{f}$ for some $P\in P_{<O_{k,m}(1)}(G,X,S^1)$ and a measurable $\tilde{f}:Z_{<k-1}(X)\rightarrow S^1$ where $\pi:X\rightarrow Z_{<k-1}(X)$ is the factor map.
\end{thm}
\begin{rem}
	Note that we do not require that the function $f:G\times X\rightarrow S^1$ is a cocycle. This theorem is a counterpart of \cite[Theorem 4.5]{Berg& tao & ziegler}
\end{rem}
\subsection{Reduction to a finite $U$}
Now we turn to the proof of Theorem \ref{TDCL:thm} assuming the induction hypothesis of Theorem \ref{Main:thm}.\\
Just like in \cite{Berg& tao & ziegler} we first show that it is suffice to prove the theorem in the case where the group $U$ is a finite group.\\

We recall the following results from \cite{Berg& tao & ziegler}
\begin{lem}  [Decent of type]\label{dec:lem} \cite[Lemma 5.1]{Berg& tao & ziegler} Let $Y$ be a $G$-system, let $k,m\geq 1$, and let $X=Y\times_{\rho} U$ be an ergodic extension of $Y$ be a phase polynomial cocycle $\rho:G\times Y\rightarrow U$ of degree $<m$. Let $\pi:X\rightarrow Y$ be a factor map and let $f:Y\rightarrow S^1$ be a function such that $\pi^\star f$ is of type $<k$, then $f$ is of type $<k+m+1$.
\end{lem}

\begin{lem} [Polynomial integration lemma] \cite[Lemma B.6]{Berg& tao & ziegler}\label{PIL:lem}
	Let $m,k\geq 1$, let $X=Y\times_{\rho} U$ be an ergodic abelian extension of a $G$-system $Y$ by a cocycle $\rho:G\times Y\rightarrow U$ that is also a phase polynomial of degree $<k$. For all $u\in U$ let $q_u: X\rightarrow S^1$ be a phase polynomial of degree $<m$ which obeys the cocycle identity $q_{uv}=q_u V_u q_v$ for all $u,v\in U$. Then there exists a phase polynomial $Q:X\rightarrow S^1$ of degree $<O_{k,m}(1)$ such that $\Delta_u Q = q_u$ for all $u\in U$.
\end{lem}
The following Proposition is a corollary of Lemma \ref{PIL:lem} and it states that a Conze-Lesigne type equation can be reduced to a factor.
\begin{prop} \label{inv:prop}
	Let $X,Y,\rho$ be as in Lemma \ref{PIL:lem}. Let $f:G\times X\rightarrow S^1$ be a function that is $(G,X,S^1)$-cohomologous to a phase polynomial $p:G\times X\rightarrow S^1$ of degree $<m$, for some $m\geq 0$. Write $f=p\cdot \Delta F$, and suppose that $f$ is invariant under the action of $U$. Then there exists a phase polynomial $p':G\times X\rightarrow S^1$ of degree $<O_{k,m}(1)$ and a measurable map $F':X\rightarrow S^1$ such that $p'$ and $F'$ are invariant under the action of $U$ and $f=p'\cdot \Delta F'$. 
\end{prop}
\begin{proof}
	By assumption we can write
	
	$$f=p\cdot \Delta F$$
	for some $p\in P_{<m}(G,X,S^1)$ and $F:X\rightarrow S^1$ a measurable map. As $f$ is invariant under $U$ and the action of $U$ commutes with the $G$-action we have $$\Delta_u p \cdot \Delta \Delta_u F=1$$
	It follows that $\Delta_u F$ is a phase polynomial of degree $<m+1$ (In fact of degree $<m$ using Lemma \ref{vdif:lem}). Moreover it satisfies the cocycle identity $\Delta_{uv} F = \Delta_u F V_u \Delta_v F$. Therefore, applying Lemma \ref{PIL:lem} we conclude that there exists a phase polynomial $Q:G\times X\rightarrow S^1$ of degree $<O_{k,m}(1)$ with the property that $\Delta_u F = \Delta_u Q$. Let $p'=p\cdot \Delta Q$ and $F'=F/Q$ we have
	$$f= p'\cdot \Delta F'$$
	As $F'$ is invariant under $U$ and $f$ is invariant under $U$ we have that $p'$ is invariant under $U$ and this completes the proof.
	
\end{proof}
We prove the following counterpart of Proposition 6.1 from \cite{Berg& tao & ziegler}. Given an integer $k$ we assume (inductively) that Theorem \ref{Main:thm} holds for smaller values of $k$ we have that
\begin{thm} \label{MainF:thm}
	In order to prove Theorem \ref{TDCL:thm} it suffices to do so in the case where $U$ is finite.
\end{thm}
\begin{proof}
	Let $X$ be a totally disconnected system of order $<k$ and write $X=Z_{<k-1}(X)\times_{\sigma} U$. By Theorem \ref{TDisweyl} we can assume that $\sigma:G\times Z_{<k-1}(X)\rightarrow U$ is a phase polynomial of degree $<O_{k}(1)$. Let $f:G\times X\rightarrow S^1$ be a function of type $<m$ such that for all $u\in U$ we have that $\Delta_u f=p_u\cdot \Delta F_u$ for some $p_u\in P_{<m}(G,X,S^1)$ and $F_u\in\mathcal{M}(X,S^1)$. From the linearization Lemma (Lemma \ref{lin:lem}) we know that there exists an open neighborhood $U'$ of the identity in $U$ on which $p_u$ is a cocycle in $u$. $U$ is totally disconnected, therefore by Theorem \ref{TST:thm} it is isomorphic to a direct product of finite cyclic groups. We conclude that $U'$ contains a cylinder neighborhood. In other words, we can replace $U'$ with an open subgroup such that $U=U'\times W$ for some finite group $W$.\\
	
	We pass from $U$ to $U'$. First we write $X=Y\times_{\sigma'} U'$ for $Y=Z_{<k-1}(X)\times_{\sigma''} W$ where $\sigma'$ and $\sigma''$ are the projection of $\sigma$ to $U'$ and $W$. By construction of $U'$ we have that $p_{uv}=p_u V_u p_v$, for all $u,v\in U'$. We now use the polynomial integration Lemma (Lemma \ref{PIL:lem}) once for every $g\in G$ and then gluing together. We get that there exists a phase polynomial $Q:G\times X\rightarrow S^1$ of degree $<O_k(1)$ such that $\Delta_u Q = p_u$ for all $u\in U'$. In particular for every $u\in U'$ we have that $\Delta_u (f/Q)$ is a $(G,X,S^1)$-coboundary. Therefore by Lemma \ref{cob:lem} we conclude that $f/Q$ is $(G,X,S^1)$-cohomologous to a function $f':G\times X\rightarrow S^1$ that is invariant under the action of some open subgroup $U''$ of $U'$.\\
	Let $\varphi:U\rightarrow U/U''$ be the quotient map and let $X'=Z_{<k-1}(X)\times_{\varphi\circ\sigma} U/U''$ with the factor map $\pi:X\rightarrow X'$. We can then write  $f'=\pi^\star\tilde{f}$ where $\tilde{f}:G\times X'\rightarrow S^1$.\\
	Since $f$ is of type $<m$ and $Q$ is a phase polynomial of degree $<O_{k,m}(1)$ we have that $f/Q$ is of type $<O_{k,m}(1)$ (Lemma \ref{PP} (iii)). By construction $\pi^\star \tilde{f}$ is $(G,X,S^1)$-cohomologous to $f/Q$ and therefore it is also of type $<O_{k,m}(1)$. Using Lemma \ref{dec:lem} we have that $\tilde{f}$ is of type $<O_{k,m}(1)$. Now, since $Q$ is a phase polynomial of degree $<O_{k,m}(1)$ we have that $\Delta_u (f/Q)$ is $(G,X,S^1)$-cohomologous to a phase polynomial. In particular, $\pi^\star \Delta_{\varphi(u)} \tilde{f}=\Delta_u \pi^\star \tilde{f}$ is $(G,X,S^1)$-cohomologous to a phase polynomial of degree $<O_{k,m}(1)$. Thus for every $u\in U/U''$ we can write $\pi^\star \Delta_u\tilde{f}=q_u\cdot \Delta F_u$ for some phase polynomial $q_u$ of degree $<O_{k,m}(1)$. Applying Proposition \ref{inv:prop} we can assume that $q_u,F_u$ are invariant under $U''$, therefore $\Delta_u\tilde{f}$ is $(G,X',S^1)$-cohomologous to a phase polynomial of degree $<O_{k,m}(1)$. Now, applying Theorem \ref{TDCL:thm} for the system $X'$, which is an extension of $Z_{<k-1}(X')$ by a finite group $U/U''$. We conclude that $\tilde{f}$ is $(G,X',S^1)$- cohomologous to $P\cdot \tilde{\pi}^\star f_0$ where $P\in P_{<O_{k,m}(1)}(G,X',S^1)$, $f_0:Z_{<k-1}(X)\rightarrow S^1$ is a measurable map and $\tilde{\pi}:X\rightarrow Z_{<k-1}(X)$ is the factor map. As $P$ is a phase polynomial of degree $<O_{k,m}(1)$ and $\tilde{f}$ is of type $<O_{k,m}(1)$, arguing as before we have that $f_0$ is also of type $<O_{k,m}(1)$.
	Finally, applying Theorem \ref{Main:thm} for the system $Z_{<k-1}(X)$, we have that $f_0$ is $(G,Z_{<k-1}(X),S^1)$-cohomologous to a phase polynomial of degree $<O_{k,m}(1)$. Lifting everything up, we conclude that $f$ is $(G,X,S^1)$-cohomologous to a phase polynomial of degree $<O_{k,m}(1)$.
\end{proof}
\subsection{Proving the theorem for a finite $U$}
We prove the following counterpart of Proposition 7.1 from \cite{Berg& tao & ziegler} for totally disconnected systems.
\begin{thm} [Theorem \ref{TDCL:thm} for a finite $U$]\label{FC:thm}
	Let $k\geq 1$ and $U$ be a finite group. Suppose that $X=Z_{<k-1}(X)\times _{\rho} U$ is a totally disconnected ergodic $G$-system of order $<k$. Let $f:G\times X\rightarrow S^1$ be a function of type $<m$ for some $m\in\mathbb{N}$ such that
	$$\Delta_t f \in P_{<m}(G,X,S^1)\cdot B^1(G,X,S^1)$$
	for all automorphism $t:X\rightarrow X$. Then $f$ is $(G,X,S^1)$-cohomologous to $P\cdot\pi^\star \tilde{f}$ for some $P\in P_{<O_{k,m}(1)}(G,X,S^1)$ and a measurable $\tilde{f}:Z_{<k-1}(X)\rightarrow S^1$ where $\pi:X\rightarrow Z_{<k-1}(X)$ is the factor map.
\end{thm}

\begin{proof}
	Finite groups are trivially totally disconnected and so by Theorem \ref{TST:thm} we have that $U$ is isomorphic to $\prod_{i=1}^N C_{{p_i}^{n_i}}$, for some $n_1,n_2,...,n_N = O_k(1)$ and $N$ unbounded but finite. Moreover if $p_i$ is sufficiently large with respect to $k$ we can take $n_i=1$. \\
	Let $e_1,...,e_N$ be the standard basis for $U$ and we write 
	\begin{equation}\label{fe1}
	\Delta_{e_i} f = Q_i \cdot \Delta F_i
	\end{equation}
	For a phase polynomial $Q_i\in P_{<m}(G,X,S^1)$ and $F_i:X\rightarrow S^1$ a measurable map.\\
	we have the telescoping identity
	$$\prod_{t=0}^{p_j^{n_j}-1} V_{e_j}^t\Delta_{e_j} f = 1$$
	and so equation (\ref{fe1}) implies that 
	$$\Delta \prod_{t=0}^{p_j^{n_j}-1} V_{e_j}^tF_j = \prod_{t=0}^{p_j^{n_j}-1}V_{e_j}^t \overline{Q_j}\in P_{<m}(G,X,S^1)$$
	In particular we have that $$\prod_{t=0}^{p_j^{n_j}-1} V_{e_j}^tF_j\in P_{<m+1}(X,S^1)$$
	We now clean up this term. We claim that there exists a phase polynomial $\psi_j$ of degree $<O_{k,m}(1)$ which is invariant under $e_j$ and is a ${p_j}^{n_j}$'th root of $\prod_{t=0}^{p_j^{n_j}-1} V_{e_j}^tF_j$. \\
	First we notice that $\prod_{t=0}^{p_j^{n_j}-1} V_{e_j}^t F_j$ is invariant under $e_j$. Now, if $p_j=O_{k,m}(1)$, then we view $\prod_{t=0}^{p_j^{n_j}-1} V_{e_j}^t F_j$ as a phase polynomial on the factor induced by quotienting out $\left<e_j\right>$. Then we apply Corollary \ref{roots} and find a phase polynomial of degree $<O_{k,m}(1)$ that is also a $p_j^{n_j}$'th root of $\prod_{t=0}^{p_j^{n_j}-1} V_{e_j}^t F_j$. Lifting everything up we conclude that there exists a phase polynomial $\psi_j$ of degree $<O_{k,m}(1)$ that is invariant under $e_j$ and is a $p_j^{n_j}$'th root of $\prod_{t=0}^{p_j^{n_j}-1} V_{e_j}^t F_j$.\\
	Otherwise if $p_j$ is sufficiently large then we can take $n_j=1$. Now, the phase polynomial $\prod_{t=1}^{p_j-1} V_{e_j}^t F_j$ is defined on a totally disconnected system and so by Proposition \ref{TDPV:prop} we have that up to a constant multiple, $\prod_{t=1}^{p_j-1} V_{e_j}^t F_j$ takes values in a finite subgroup $H$ of $S^1$. Rotating $F_j$ by a $p_j$'th root of this constant we may assume that $\prod_{t=1}^{p_j-1} V_{e_j}^t F_j$ takes values in $H$ without changing equation (\ref{fe1}). Let $\pi_i:H\rightarrow C_{q_i^{l_i}}$ be one of the coordinate maps of $H$, we study the term $\pi_i\circ\prod_{t=1}^{p_j-1} V_{e_j}^t F_j$. If $q_i\not = p_j$ then $p_j$ is invertible with respect to the multiplicative operation on $\mathbb{Z}/q_i^{l_i}\mathbb{Z}$. We conclude that some power of  $\prod_{t=1}^{p_j-1} V_{e_j}^t \pi_i\circ F_j$ is a $p_j$'th root that is also a phase polynomial of degree $<m+1$ and is invariant under $V_{e_j}$, let $\psi_{i,j}$ denote this root. Otherwise if $q_i=p_j$ then $q_i$ is also sufficiently large, Theorem \ref{HTDPV:thm} implies that $l_i=1$. Observe that the derivative of $\pi_i\circ\prod_{t=1}^{p_j-1}V_{e_j}^t F_j$ satisfies
	
	\begin{equation} \label{fe1.1}
	\pi_i\circ\prod_{t=0}^{p_j-1} V_{e_j}^t Q_j = \pi_i\circ\prod_{t=0}^{p_j-1} \left( \Delta_{e_j} ^t  Q_j\right)^{\binom{p_j}{t+1}}
	\end{equation}
	
	Write $G=G_{p_j} \oplus G'$ where $G_{p_j}$ is the subgroup of elements of order $p_j$, and $G'$ be it's complement. The terms in equation (\ref{fe1.1}) are of order $q_i=p_j$ and therefore if $g\in G'$ then by Proposition \ref{PPC} both of them are trivial. For $g\in G_{p_j}$ we claim that the right hand side is trivial. 
	For such $g\in G_{p_j}$ we have by Proposition \ref{PPC} and Theorem \ref{HTDPV:thm} that $Q_j(g,\cdot)$ takes values in $C_{p_j}$. Observe that $p_j$ divides $\binom{p_j}{t+1}$ for all $0\leq t< p_j-1$, and since $p_j$ is sufficiently large  $\Delta_{e_j}^{p_j-1}$ eliminates $Q_j$ (by Lemma \ref{vdif:lem}), hence $\prod_{t=0}^{p_j-1} \left( \Delta_{e_j} ^t  Q_j\right)^{\binom{p_j}{t+1}}=1$. Since $Q_j$ is a cocycle we conclude that the terms in equation (\ref{fe1.1}) are trivial for every $g\in G$. Therefore by ergodicity $\pi_i\circ\prod_{t=1}^{p_j-1} V_{e_j}^t F_j$ is a constant and so we can find a $p_j$'th root which we denote by $\psi_{i,j}$.\\
	
	Thus, in any case every coordinate of $\prod_{t=1}^{p_j^{n_j}-1} V_{e_j}^t\circ F_j$ has a root that is also a polynomial of degree $<m+1$ and is invariant under $V_{e_j}$. Gluing all the coordinates $\psi_{i,j}$ together, we see that there exists a phase polynomial $\psi_j(x)$ which is the product of all $\psi_{i,j}(x)$ of degree $<m+1$ that is a $p_j^{n_j}$'th root of $\prod_{t=1}^{p_j^{n_j}-1} V_{e_j}^t F_j$. This proves the claim.\\
	
	Now set $\tilde{F}_j  := F_j/\psi_j$ and $\tilde{Q}_j := Q_j\cdot \Delta \psi_j$ then $\tilde{Q}_j$ is a phase polynomial of degree $<O_{k,m}(1)$ and we have
	
	\begin{equation}\label{fe2}
	\Delta_{e_j} f = \tilde{Q}_j\cdot \Delta \tilde{F}_j
	\end{equation}
	and
	\begin{equation} \label{fe3}
	\prod_{t=1}^{p_j^{n_j}-1} V_{e_j}^t \tilde{F}_j=1
	\end{equation}
	Write $[t_1,...,t_N]=e_1^{t_1}\cdot...\cdot e_N^{t_N}$ and $X=Z_{<k-1}(X)\times_{\sigma} U$ and let $F:X\rightarrow S^1$ be given by 
	$$F(y,[t_1,...,t_N]) = \prod_{i=1}^N\prod_{0\leq t_i'<t_i} \tilde{F}_i(y,[t_1,...,t_{i-1},t_i',0,...,0])$$
	with the convention that $\prod_{0\leq t_i' < t_i} a_{t_i'}=\left(\prod_{t_i<t_i'\leq 0} a_{t_i'}\right)^{-1}$. Equation (\ref{fe2}) implies that $F$ is well defined.\\
	We compute the derivatives of $F$. We have,
	$$\Delta_{e_j}F(y,[t_1,...,t_N])=\prod_{i=1}^N \prod_{0\leq t_i'<t_i} \Delta_{e_j} F_j(y,[t_1,...,t_{i-1},t_i',0,...,0])$$
	For any $1\leq j\leq N$. On the other hand, we have that telescoping identity
	\[
	\prod_{j=1}^N \prod_{0\leq t_j'<t_j} \Delta_{e_j} F_i(y,[t_1,...,t_{j-1},t_j',0,...,0]) = \frac{F_i(y,[t_1,...,t_N])}{F_i(y,1)}
	\]
	and thus,
	\begin{equation}\label{fe4}
	\Delta_{e_j}F(y,[t_1,...,t_N])=\frac{\tilde{F}_j(y,[t_1,...,t_N])}{\tilde{F}_j(y,1)}\prod_{i=1}^N \prod_{0\leq t_i'<t_i} \omega_{i,j}(y,[t_1,...,t_{i-1},t_i',0,...,0])
	\end{equation}
	where $\omega_{i,j}=\frac{\Delta_{e_i}\tilde{F}_j}{\Delta_{e_j}\tilde{F}_i}$. Note that from (\ref{fe2}) we have that $\omega_{i,j}$ is a phase polynomial of degree $<O_{k,m}(1)$. \\
	
	As $\sigma$ is a phase polynomial of degree $<O_k(1)$ we have that $(y,u)\mapsto u$ is a phase polynomial of degree $<O_k(1)$, which implies that $(y,[t_1,...,t_n])\rightarrow [t_1,...,t_{i-1},t_i',0,...,0]$ is also a phase polynomial of degree $<O_k(1)$ as $[t_1,...,t_n]\mapsto[t_1,...,t_{i-1},t_i',0,...,0]$ is a constant multiple of a homomorphism. From Lemma \ref{B.5} we have that $(y,[t_1,...,t_n])\mapsto \omega_{i,j}(y,[t_1,...,t_{i-1},t_i',0,...,0])$ are all phase polynomials of degree $<O_{k,m}(1)$.\\
	We claim that the map $$\xi_j(y,[t_1,...,t_N]) =\prod_{i=1}^N\prod_{0\leq t_i'<t_i}\omega_{i,j}(y,[t_1,...,t_{i-1},t_i',0,...,0])$$ is a phase polynomial of degree $<O_{k,m}(1)$. Clearly $\xi_j$ is a function of phase polynomials. From Theorem \ref{TDPV:thm} we see that there exists a finite subgroup $H\leq S^1$ such that all $\omega_{i,j}(y,[t_1,...,t_{i-1},t_i',...])$ takes values in $H$ (up to a constant multiple). Therefore so does $\xi_j$. Let $\pi_n:H\rightarrow C_{q_n^{l_n}}$ be one of the coordinate maps, now if $q_n=O_{k,m}(1)$ then Corollary \ref{roots} implies that $\pi_n\circ\xi_j$ is a phase polynomial of degree $<O_{k,m}(1)$. Otherwise if $q_n$ is sufficiently large then $\pi_n\circ \omega_{i,j}(y,[t_1,...,t_{i-1},t_i',0,...,0])$ is a phase polynomial in $t_i'$ that, by Theorem \ref{HTDPV:thm}, takes values in $C_{q_n}$. By Taylor expansion we may thus write
	
	$$ \omega_{i,j}(y,[t_1,...,t_{i-1},t_i',0,...,0])=\prod_{0\leq j \leq O_{k,m}(1)} \left[ \Delta_{e_i}^j \omega_{i,j}(y,[t_1,...,t_{i-1},t_i',0,...,0])\right]^{\binom{t_i'}{j}}$$
	and thus
	$$\prod_{0\leq t_i'\leq t_i} \omega_{i,j}(y,[t_1,...,t_{i-1},t_i',0,...,0])=\prod_{0\leq j \leq O_{k,m}(1)} \left[ \Delta_{e_i}^j \omega_{i,j}(y,[t_1,...,t_{i-1},t_i',0,...,0])\right]^{\binom{t_i}{j+1}}$$
	Lemma \ref{B.5} then implies that $\pi_n\circ\xi_j$ is also a phase polynomial of degree $<O_{k,m}(1)$. We conclude that $\xi_j\in P_{<O_{k,m}(1)}(X)$, from (\ref{fe4}) we have that
	$$\Delta_{e_j} F (y,u) = \frac{\tilde{F}_j(y,u)}{\tilde{F}_j(y,1)}\cdot P_{<O_{k,m}(1)}(X,S^1)$$
	Now let $f'=f/\Delta F$ then 
	\begin{equation}\label{eqp}
	\Delta_{e_j}f' \in \left(\pi^\star F_j'\right)\cdot P_{<O_{k,m}(1)}(G,X,S^1) 
	\end{equation}
	where $\pi^\star F_j' (y,u) = \tilde{F}_j (y,1)$.\\
	We repeat the same argument as above now with $f'$ instead, we have the telescoping identity
	$$\prod_{t=0}^{p_j^{n_j}-1}V_{e_j}^tf'=1$$
	Which implies that
	$$\pi^\star \Delta (F_j')^{p_j^{n_j}}\in P_{<O_{k,m}(1)}(G,X,S^1)$$ Therefore $(F_j')^{p_j^{n_j}}$ is a phase polynomial of degree $<O_{k,m}(1)$ on $Z_{<k-1}(X)$.\\
	We apply the same procedure as before, we see that there exists a $p_j^{n_j}$'th root $P_j$ for the phase polynomial $(F_j')^{p_j^{n_j}}$. Let $F_j'' = F_j'/P_j$, then $F_j''$ takes values in $C_{p_j^{n_j}}$ and from (\ref{eqp}) we have $$(\Delta_{e_j} f') \in \pi^\star F_j'' \cdot P_{<O_{k,m}(1)}(G,X,S^1)$$
	We define $F^\star :X\rightarrow S^1$ to be the function $F^\star(y,[t_1,...,t_N]):=\prod_{j=1}^N F_j''(y)^{t_j}$; $F^\star$ is well defined since $F_j''$ takes values in $C_{{p_j}^{n_j}}$. We observe that $\pi^\star \Delta F_j'' =\Delta_{e_j}\Delta F^\star$. Let $f''=f'/\Delta F^\star$, then $f''$ is cohomologous to $f$ and $\Delta_{e_j}f''\in P_{<O_{k,m}(1)}(G,X,S^1)$ for all $1\leq j \leq N$. The cocycle identity implies that $\Delta_u f''\in P_{<O_{k,m}(1)}(G,X,S^1)$ for every $u\in U$.  We integrate this term using Lemma \ref{PIL:lem} once for every $g\in G$, we have that $\Delta_u f'' = \Delta_u P$ for some $P\in P_{<O_{k,m}(1)}(G,X,S^1)$. It follows that $f''/P$ is invariant under $U$ and so $f''=P\cdot \pi^\star \tilde{f}$ where $\tilde{f}:G\times Z_{<k-1}(X)\rightarrow S^1$.
\end{proof}

\section{The High Characteristic Case} \label{high:sec}
Throughout this section, we denote $\text{char} (G) = \min\{p:p\in\mathcal{P}\}$. We prove the following version of Theorem \ref{Main:thm}
\begin{thm}\label{reductionH:thm}
	Let $1\leq k,j \leq \text{char}(G)$ and $X$ be an ergodic $G$-system of order $<j$. Let $\rho:G\times X\rightarrow S^1$ be a cocycle of type $<k$, then $\rho$ is $(G,X,S^1)$-cohomologous to a phase polynomial of degree $<k$.
\end{thm}
As with the original version, we want to prove this theorem in two steps. First we want a reduction to a totally disconnected case and then we will solve this case.\\
We begin by introducing some definitions and an important Lemma from \cite{Berg& tao & ziegler}
\begin{defn} [Quasi-cocycles]
	Let $X$ be an ergodic $G$-system, let $f:G\times X\rightarrow S^1$ be a function. We say that $f$ is a quasi-cocycle of order $<k$ if for every $g,g'\in G$ one has $$f(g+g',x)=f(g,x)\cdot f(g',T_gx)\cdot p_{g,g'}(X)$$ for some $p_{g,g'}\in P_{<k}(X,S^1)$. 
\end{defn}
\begin{lem} [Exact decent] \label{Edec:lem} Let $X$ be an ergodic $G$-system of order $<k$ for some $k\geq 0$. Let $\pi:X\rightarrow Y$ be a factor map. Suppose that a function $f:G\times Y\rightarrow S^1$ is a quasi-cocycle of order $<k$. If $\pi^\star f$ is of type $<k$ then so is $f$.
\end{lem}
\begin{defn} [Line-cocycle]\label{lc:def} 
	Let $X$ be an ergodic $G$-system and let $f:G\times X\rightarrow S^1$ be a function. We say that $f$ is a line cocycle, if for every $g\in G$ of order $n$ we have $\prod_{t=0}^{n-1} f(g, T_g^tx) = 1$
\end{defn}

We claim that Theorem \ref{reductionH:thm} follows from
\begin{thm} (Reduction to totally disconnected case in High characteristics) \label{TDH:thm}
	Let $1\leq j,k \leq \text{char} (G)$. Let $X$ be a totally disconnected and Weyl ergodic $G$-system of order $<j$ where the phase polynomial cocycles $\sigma_1,...,\sigma_j$ that defines $X$ are of degrees $<1,...,<j$. Then every $f:G\times X\rightarrow S^1$ of type $<k$ which is also a line cocycle and a quasi-cocycle of order $<k-1$, is $(G,X,S^1)$-cohomologous to a phase polynomial $P$ of degree $<k$. Moreover, for $g\in G$ of order $n$ we have that $P(g,\cdot)$ takes values in $C_n$.
\end{thm}

\begin{proof}[Proof of Theorem \ref{reductionH:thm} assuming Theorem \ref{TDH:thm}]
	Let $k,j,X,\rho$ be as in Theorem \ref{reductionH:thm}, we prove this theorem by induction on $k$. As $\rho$ is of finite type, applying Theorem \ref{Main:thm} we have that $\rho$ is cohomologous to a phase polynomial of some degree. Therefore we can apply Theorem \ref{con:thm} to eliminate the connected components of $X$. As in Theorem \ref{TDred:thm}, $X$ admits a totally disconnected and Weyl factor $Y$, where the phase polynomial cocycles $\sigma_1,...,\sigma_l$ are of degree $<1,...,<l$. And $\rho$ is cohomologous to $\pi^\star \rho'$ where $\rho':G\times Y\rightarrow S^1$. By Lemma \ref{Cdec:lem} we have that $\rho'$ is of type $<k$ and therefore by Theorem \ref{TDH:thm} we have that it is cohomologous to a phase polynomial $P$ of degree $<k$. We conclude that $\rho$ is cohomologous to $\pi^\star P$ which is also a phase polynomial of degree $<k$.
\end{proof}
It is left to prove Theorem \ref{TDH:thm}, the proof is going to be very similar to \cite[Theorem 8.6]{Berg& tao & ziegler}. We first deal with the easy case $k=1$. In this case $f$ is a quasi-cocycle of order $<0$, and is thus a cocycle. It is well known that a cocycle of type $<1$ is cohomologous to a constant $c(g)$ (See \cite[Lemma C.5]{HK} or \cite[Lemma B.9]{Berg& tao & ziegler} or \cite{MS}), so we can write $f(g,x)=c(g)\cdot \Delta_g F(x)$ for some measurable $F:X\rightarrow S^1$. Since $f$ is a cocycle, $c$ is a character of $G$, and therefore $c(g)\in C_n$ for $g$ of order $n$ and the claim follows.

Now suppose that $2\leq k \leq \text{char}(G)$ and assume inductively that the claim has already been proven for smaller values of $k$. We have an analogous of Theorem \ref{TST:thm}
\begin{thm}  [Exact topological structure theorem for totally disconnected systems] \label{TSTH:thm}Let $1\leq k \leq \text{char}(G)$ be such that Theorem \ref{TDH:thm} holds for all values smaller or equal to $k$. Let $X$ be an ergodic totally disconnected Weyl  $G$-system of order $<k$ where the phase polynomial cocycles $\sigma_1,...,\sigma_{k-1}$ are of degree $<1,...,<k-1$, then we can write $Z_{<j}(X)=Z_{<j-1}(X)\times _{\sigma_{j-1}} U_j$, where $U_j \cong \prod_{p\in A} C_p$ 
\end{thm}
When $j=1$, $X$ is just a point and Theorem \ref{TDH:thm} is trivial. Now suppose $2\leq j\leq \text{char}(G)$ and assume inductively the claim has already been proven for the same value of $k$ and smaller values of $j$.

We first deal with the lower case $j\leq k$, write $X=U_0\times_{\sigma_1}U_1\times...\times_{\sigma_{j-1}} U_{j-1}$ and let $t\in U_{j-1}$ we have
\begin{lem} \label{ED:lem} \cite[Lemma 8.8]{Berg& tao & ziegler}
	$\Delta_t f$ is a line cocycle, is of type $<k-j+1$  and is a quasi-coboundary of order $<k-j$.
\end{lem}
By the induction hypothesis, $\Delta_t f$ is $(G,X,S^1)$-cohomologous to a phase polynomial $q_t\in P_{<k-j+1}(G,X,S^1)$, and $q_t(g,\cdot)$ takes values in $C_n$ for $g$ of order $n$. Since $\Delta_t f$ is a line cocycle and a quasi-cocycle of order $<k-j$ so is $q_t$.
\subsection{Reduction to the finite $U$ case} We now argue as in Theorem \ref{MainF:thm}. We show that is suffices to prove \ref{TDH:thm}, in the case when $U_{j-1}$ is finite.\\
We will take advantage of the following result of Bergelson Tao and Ziegler.
\begin{lem}  [Exact integration Lemma]\label{Eint:lem}\cite[Proposition 8.9]{Berg& tao & ziegler}
	Let $j \geq 0$, let $U$ be a compact abelian group, and let $X = Y \times_{\rho} U$ be an ergodic $G$-system with $Y \geq Z_{<j}(X)$, where $\sigma:G\times Y\rightarrow U$ is a phase polynomial cocycle of degree $<j$. For any $t\in U$, let $p_t:X\rightarrow S^1$ be a phase polynomial of degree $< l$, and suppose that for any $t,s\in U$, $p_{t\cdot s} = p_s(x) p_t(V_sx)$. Then there exists a phase polynomial $Q:X\rightarrow S^1$ of degree $<l+j$ such that $\Delta_t Q = p_t(x)$. Furthermore we can take $Q(t,uu_0):=p_u (y,u_0)$ for some $u_0\in U$.
\end{lem}
We write $\Delta_t f = q_t\cdot \Delta F_t$, by Lemma \ref{lin:lem} (linearization Lemma) there exists an open neighborhood of the identity $U_{j-1}'$ in $U_{j-1}$ on which $q_t$ is a cocycle in $t$. As in Theorem \ref{MainF:thm} we can write $U_{j-1} = U'\times W$ for some finite $W$ and write $X = Y\times _{\sigma'} U'$ where $Y=Z_{<j-1}(X)\times_{\sigma''} W$ and $\sigma',\sigma''$ are the projections of $\sigma_{j-1}$. Note that as $\sigma_{j-1}$ is of degree $<j-1$, $\sigma'$ is also.\\
Applying Lemma \ref{Eint:lem} once for every $g\in G$, we can write $q_t = \Delta_t Q$ for all $t\in U_{j-1}$ and some phase polynomial such that $Q(g,y,uu_0)=q_u(g,y,u_0)$ for all $y\in Y$ and $u\in U'$ and some $u_0\in U'$. As $q_u(g,\cdot)$ take values in $C_n$ whenever $g$ is of order $n$, so is $Q$.\\

We now claim that $Q:G\times X\rightarrow S^1$ is a quasi-cocycle of order $<k-1$. Indeed, for any $g,h\in G$ and $x=(y,uu_0)\in X$, we have
$$\frac{Q(g+h,x)}{Q(g,x)Q(h,T_gx)} = \frac{Q(g+h,x)}{Q(g,x)Q(h,x)\Delta_g Q(h,x)}$$
$$=\frac{q_u(g+h,y,u_0)}{q_u(g,y,u_0)q_u(h,y,u_0)\Delta_g Q(h,x)}$$
$$=\frac{q_u(g+h,y,u_0)}{q_u(g,yu,u_0)q_u(h,T_gy,\sigma_{j-1}(g,y)u_0)}\frac{q_u(h,T_gy,\sigma_{j-1}(g,y)u_0)}{q_u(h,y,u_0)Q(h,x)}$$
$$=P_{u,g,h}(y,u_0)\Delta_g \frac{q_u(h,x)}{Q(h,x)}$$
where $P_{u,g,h}(x)=\frac{q_u(g+h,x)}{q_u(g,x)q_u(h,T_g x)}$. The term $\Delta_g \frac{q_u(h,x)}{Q(h,x)}$ is a phase polynomial of degree $<k-1$. As $q_u$ is a quasi-cocycle of order $<k-j$ we have that $P_{u,g,h}(x)$ is a phase polynomial of degree $<k-j$. $u\mapsto q_u$ is a cocycle in $u$ and therefore so is $u\mapsto P_{u,g,h}$. Thus, applying Lemma \ref{Edec:lem} we can integrate $P_{u,g,h}$ applying Lemma \ref{vdif:lem} we have that $P_{u,g,h}(y,u_0)$ is also a phase polynomial of degree $<k-1$ and the claim follows.\\
Now write $f'=f/Q$ and apply Lemma \ref{cob:lem} we have that $f'$ is $(G,X,S^1)$ to $f''$ which is invariant under some open subgroup $U''$ of $U$. Let $\pi:Y\times_{\overline\sigma_{j-1}} U_{j-1}/U''$ be the factor map induced by the projection $U_{j-1}\rightarrow U_{j-1}/U''$ we can write $f'' = \pi^\star \tilde{f}$ for some $\tilde{f}:G\times Y\times_{\overline\sigma_{j-1}} U_{j-1}/U''\rightarrow S^1$. Since $Q(g,\cdot)$ takes values in $C_n$ for $g$ of order $n$ and is a phase polynomial whose degree is smaller than any prime divides $n$, by Proposition \ref{linecocycle:prop} we have that $Q$ is a line-cocycle. As $\pi^\star \tilde{f}$ is cohomologous to $f/Q$, we conclude that $\pi^\star \tilde{f}$ and hence $\tilde{f}$ are line cocycles. Similarly as $f,Q$ are quasi-cocycles of order $<k-1$. $\pi^\star \tilde{f}$ is also. \\
Lastly, as $Q$ is a phase polynomial of degree $<k$ it is of type $<k$. Since $f$ is also of type $<k$ we conclude that $\pi^\star \tilde{f}$ is of type $<k$ applying Lemma \ref{Edec:lem} we have that $\tilde{f}$ is of type $<k$. 
We can now apply the finite case of Theorem \ref{TDH:thm} to complete the proof.
\subsection{The finite group case}
We now establish Theorem \ref{TDH:thm} for a finite $U_j$. \\
First we recall the following Lemma 
\begin{lem} [Free actions of compact abelian groups have no cohomology] \cite[Lemma B.4]{Berg& tao & ziegler}, \cite[Lemma C.8]{HK} \label{B.4}
	Let $U$ be a compact abelian group acting freely on $X$ by measure preserving transformations. Then every cocycle $\rho:U\times X\rightarrow S^1$ is a $(U,X,S^1)$-coboundary.
\end{lem}

Using Theorem \ref{TSTH:thm} we can write $U_j = C_{p_1}^{L_1}\times...\times C_{p_m}^{L_m}$ for some finite $L_1,...,L_m$ and a finite $m\in\mathbb{N}$. We proceed by induction on $L=L_1+...+L_m$, the case $L=0$ is trivial. Suppose that $L\geq 1$, then without loss of generality $L_1\geq 1$ and we assume inductively that the claim has already been proven for $L-1$. Write $U_j = C_{p_1}^{L_1-1}\times \left<e\right>\times C_{p_2}^{L_2}\times... C_{p_m}^{L_m}$ where $e$ is a generator for $C_{p}$ where $p=p_1$.

Recall that $\Delta_e f$ is $(G,X,S^1)$-cohomologous to a phase polynomial $q_e$ of degree $<k-j+1$, such that $q_e(g,\cdot)$ takes values in $C_n$ where $n$ is the order of $g$. We write,
\begin{equation}\label{Heq1}
\Delta_e f = q_e \cdot \Delta F_e
\end{equation}
Arguing as in Theorem \ref{FC:thm} we can assume that $F_e$ satisfies that \begin{equation}\label{Heq2}
\prod_{i=0}^{p-1} V_e^i F_e = 1
\end{equation} without changing equation (\ref{Heq1}) and $q_e$ is still a phase polynomial of degree $<k-j+1$.\\
We now define a function $q_{e^s}(g,x)$ for all $0\leq s <p$ by the formula $$q_{e^s}(g,x) :=  \prod_{i=0}^{s-1} q_e(g, V_{e}^i x)$$ Since $q_e$ is a phase polynomial of degree $<k-j+1$, $q_{e^s}$ is also, and by equation (\ref{Heq1}) and (\ref{Heq2}) we have that $\prod_{i=0}^{p-1} V_e^i q_e = 1$ and so $u\mapsto q_u$ is a cocycle for $u\in \left<e\right>$.\\
Since $\Delta_e f$ is cohomologous to $q_e$, by the cocycle identity we have
\begin{equation}
\Delta_{e^s} f (g,x)= \prod_{i=0}^{s-1}\Delta_{e} f(g,V_e^ix)
\end{equation}
It follows that $\Delta_u f$ is cohomologous to $q_u(g,x)$ for all $u\in \left<e\right>$.\\
Now, as $q_u$ is cocycle in $u$, applying the Polynomial integration lemma (Lemma \ref{Eint:lem}) there exists a phase polynomial $Q$ of degree $<k$ such that $\Delta_u Q = q_u$, and $Q(g,\cdot)$ take values in $C_n$ where $n$ is the order of $g$. Moreover, arguing as in the previous section we also have that $Q$ is a quasi-cocycle of order $<k-1$.\\
From equation (\ref{Heq1}) we have $$\Delta_e (f/Q) = \Delta F_e$$
We have the telescoping identity
$$\prod_{i=0}^{p-1}V_e^i \Delta_e (f/Q)=1$$ We conclude that $\Delta \prod_{i=0}^{p-1}V_{e^i}F_e=1$ and so by ergodicity $\prod_{i=0}^{p-1}V_{e^i}F_e$ is a constant in $S^1$. Thus, we can rotate $F_e$ by a $p$'th root of this constant, we can then assume that  $$\prod_{i=0}^{p-1}V_{e^i}F_e=1$$
Now, by (\ref{Heq2}) we can define $$F_{e^s} :=\prod_{i=0}^{s-1}V_e^i F_e$$ Direct computation shows that $F_u$ is a cocycle for $u\in \left<e\right>$ hence a coboundary (Lemma \ref{B.4}), thus we can write $F_e = \Delta_e F$ for some $F:X\rightarrow S^1$. We conclude that $\Delta_e (f/(Q\cdot \Delta F))=1$ and so $f/Q$ is cohomologous to $f'$ which is invariant under $\left<e\right>$, arguing as before and using the induction hypothesis we conclude that $f/Q$ is cohomologous to a phase polynomial $P$. From Theorem \ref{HTDPV:thm} we have that $P(g,\cdot)$ takes values in $C_n$ where $n$ is the order of $g$. Thus $f$ is $(G,X,S^1)$-cohomologous to $Q\cdot P$ and this completes the proof.
\subsection{The higher order case}
The case $j\geq k$ is completely analogous to the proof of Bergelson Tao and Ziegler, \cite[Section 8.5]{Berg& tao & ziegler} and so is omitted. We therefore completed the proof of Theorem \ref{MainH:thm}. 
\section{Finishing the proof of Theorem \ref{Mainr:thm}} \label{proof}
Now that Theorem \ref{Main:thm} is established the proof which we began in section \ref{sr:sec} is now complete. It is left to prove the other direction of Theorem \ref{Mainr:thm}. That is,
\begin{thm*} Let $X$ be an ergodic system of order $<k$ and suppose that  $Z_{<1}(X),...,Z_{<k}(X)$ are strongly Abramov. Then there exists a totally disconnected factor $Y$ such that for all $m\in\mathbb{N}$ the homomorphism $$\pi_l^\star : H_{<m}^1(G,Z_{<l}(Y),S^1)\rightarrow H_{<m}(G,Z_{<l}(X),S^1)$$ is onto for all $1\leq l \leq k$ where $\pi_l:Z_{<l}(X)\rightarrow Z_{<l}(Y)$ is the factor map.
\end{thm*}
\begin{proof}
 We prove the claim by induction on $k$, for $k=1$ the system $X$ is trivial and we can take $Y=X$. Let $k>1$ and suppose that the claim has already been proven for systems of order $<k-1$. Fix $1\leq l<k$ and recall that $Z_{<l+1}(X) = Z_{<l}(X)\times_{\sigma_l} U_l$ for some compact abelian group $U_l$ and a cocycle $\sigma_l:G\times Z_{<l}(X)\rightarrow U_l$ of type $<l$. Therefore by applying the induction hypothesis and Theorem \ref{Main:thm} for all $\chi\in\hat U_l$ we have that $\chi\circ\sigma_l$ is $(G,Z_{<l}(X),S^1)$-cohomologous to a phase polynomial.\\

Our main tool now is Proposition \ref{pinv:prop} which says that polynomials are invariant under connected groups. We prove by induction on $l<k$ that there exists a compact connected abelian group $\mathcal{H}_l$ which acts on $Z_{<l}(X)$ by automorphisms and contains the group of transformations $U_{l-1,0}$. For $l=1$, we let $\mathcal{H}_1$ be the trivial group. Suppose we already constructed $\mathcal{H}_l$ for some $l\geq 1$. For every $\chi\in\hat U_l$, $\chi\circ\sigma_l$ is cohomologous to a phase polynomial. Since $\mathcal{H}_l$ is connected Proposition \ref{pinv:prop} implies that $\Delta_h \chi\circ\sigma_l$ is a $(G,Z_{<l}(X),S^1)$-coboundary. Since this holds for every $\chi\in U_l$ we conclude that $\Delta_h \sigma_l$ is a $(G,Z_{<l}(X),U_l)$-coboundary.\\

For $h\in\mathcal{H}_l$ and $F:Z_{<l}(X)\rightarrow U_l$ let $S_{h,F}(x,u) := (hx,F(x)u)$. Denote by $H_{l+1}:=\{S_{h,F}:\Delta_h\sigma_l=\Delta F\}$. Clearly $H_{l+1}$ acts on $Z_{<l+1}(X)$, direct computation reveals that this action commutes with the $G$-action. Since $\mathcal{H}_{l}$ is connected, by Proposition \ref{pinv:prop} we have that $H_{l+1}$ is $2$-step nilpotent. Finally if $u\in U_l$ is a constant, then $S_{1,u}\in\mathcal{H}_l$ therefore we can view $U_l$ as a subgroup of $H_{l+1}$. Finally since $\Delta_h \sigma_l$ is a $(G,Z_{<l}(X),U_l)$-coboundary, the projection $H_{l+1}\rightarrow \mathcal{H}_l$ is onto. In other words we have a short exact sequence
$$1\rightarrow U_{l+1}\rightarrow H_{l+1}\rightarrow \mathcal{H}_l\rightarrow 1$$ since $U_{l+1}$ and $\mathcal{H}_l$ are compact so is $H_{l+1}$.\\

Now let $\mathcal{H}_{l+1}$ be the connected component of the identity in $H_{l+1}$. Then $\mathcal{H}_{l+1}$ is a compact connected nilpotent group and so is abelian (Proposition \ref{connilabel:prop}). Clearly it contains $U_{l,0}$ and it acts on $Z_{<l+1}(X)$ by automorphisms. This proves the induction step. Let $Y=X/\mathcal{H}$ where $\mathcal{H}=\mathcal{H}_k$. We return to the proof of the original claim.

Fix $1\leq l \leq k$ and $m\in\mathbb{N}$ and let $\rho:G\times Z_{<l}(X)\rightarrow S^1$ be a cocycle of type $<m$. Let $\tilde{X}:= Z_{<l}(X)\times_\rho S^1$. Since $Z_{<l}(X)$ is strongly Abramov we have that $\tilde{X}$ is Abramov of order $<l_m$. We prove that $\rho$ is cohomologous to a phase polynomial, to do this we consider two cases:\\
\textbf{Case I:} $\tilde{X}$ is ergodic. If $P:\tilde{X}\rightarrow S^1$ is a phase polynomial then by Proposition \ref{pinv:prop} we see that $\Delta_s P = \chi(s)$ for all $s\in S^1$ for some character $\chi:S^1\rightarrow S^1$. Therefore $P = \chi \cdot F$ where $F:Z_{<l}(X)\rightarrow S^1$.\\ If $\chi:S^1\rightarrow S^1$ is not the identity then $P=\chi\cdot F$ is orthogonal to the identity. Since the system is Abramov, the polynomials generate $L^2(\tilde{X})$ and therefore some polynomial must be of the form $P(x,s)=sF(x)$ for $x\in Z_{<l}(X),s\in S^1$ and $F:Z_{<l}(X)\rightarrow S^1$.
By taking derivatives we see that $$\rho \cdot \Delta F = \Delta P$$ Since $\Delta \Delta_s P = 1$ and $s$ commutes with the action of $G$, the cocycle $\Delta P$ is defines a phase polynomial on $Z_{<l}(X)$. In other words, $\rho$ is $(G,Z_{<l}(X),S^1)$-cohomologous to a phase polynomial of degree $<l_m-1$. By Proposition \ref{pinv:prop} this polynomial is measurable with respect to $Z_{<l}(X)/\mathcal{H}_l \cong Z_{<l}(Y)$. This complete the proof in the case where $\tilde{X}$ is ergodic.\\
\textbf{Case II} If $\tilde{X}$ is not ergodic. Then by the theory of Mackey $\rho$ is $(G,Z_{<l}(X),S^1)$-cohomologous to a minimal cocycle $\tau:G\times Z_{<l}(X)\rightarrow S^1$ which takes values in a proper closed subgroup of $S^1$ (See \cite[Corollary 3.8]{Zim}). In particular this means that $\tau^n = 1$ for some $n\in\mathbb{N}$. We claim by induction on $m$ that if $X$ is as in the theorem, then every $(G,X,S^1)$-cocycle of type $<m$ is cohomologous to a phase polynomial of degree $<l_m$. For $m=0$ the claim is trivial; Therefore by induction hypothesis we may assume that as a cocycle into $S^1$, $\Delta_h \tau$ is $(G,X,S^1)$-cohomologous to a phase polynomial for all $h\in \mathcal{H}_l$. Write,
$$\Delta_h \tau = p_h\cdot \Delta F_h$$ Since $\tau^n=1$ we conclude that $p_h^n$ is a coboundary. The cocycle identity and Proposition \ref{pinv:prop} implies that $p_{h^n}$ is a coboundary for all $h\in \mathcal{H}_l$. Since $\mathcal{H}_l$ is connected it is also divisible and so $\Delta_h \tau$ is a coboundary for all $h\in\mathcal{H}_l$. Finally, since $\mathcal{H}_l$ acts freely on $Z_{<l}(X)$ by automorphisms Lemma \ref{cob:lem} implies that $\tau$ is $(G,Z_{<l}(X),S^1)$-cohomologous to a cocycle that is invariant under $\tilde{H}_l$, hence measurable with respect to $Z_{<l}(Y)$ (and therefore is cohomologous to a phase polynomial by Theorem \ref{Main:thm}). Since $\tau$ and $\rho$ are cohomologous this completes the proof.
\end{proof}
\section{A $G$-system that is not Abramov} \label{example}
In this section we provide an example of a $G$-system $X$ of order $<3$ that is not an Abramov system. This example is a generalization of the Furstenberg-Weiss example for a $\mathbb{Z}$-cocycle that is not cohomologous to a polynomial (which is given in details in \cite{HK02}). Our example is constructed as an extension of a finite dimensional compact abelian group (see Definition \ref{FD:def}) that is not a Lie group by the circle group. Such groups are called solenoids and are known for their pathological properties.\\ We first construct the underline group and the $G$-action on it.
\subsection{The underline group}
Let $\mP$ denote the set of all prime numbers. We construct a $2$-dimensional compact abelian group as follows: Let $\mP_1,\mP_2$ be disjoint infinite sets such that $\mP=\mP_1\bigsqcup\mP_2$ and let $\Delta_1 := \prod_{p\in \mP_1} C_p$ and $\Delta_2:= \prod_{p\in \mP_2} C_p$. The group $\mathbb{Z}$ lies in $\Delta_1$ and in $\Delta_2$ diagonally. To see this, consider the map $1\mapsto (\omega_{p_1},\omega_{p_2},\omega_{p_3},...)$ where $\omega_{p_i}$ is the first root of unity of order $p_i$. This gives a rise to the following $1$-dimensional compact abelian groups, $U_1:=(\mathbb{R}\times \Delta_1)/\mathbb{Z}$ and $U_2:=(\mathbb{R}\times \Delta_2)/\mathbb{Z}$. We note that for every $i\in\{1,2\}$ $$\{1\}\rightarrow \Delta_1\rightarrow U_i\rightarrow \mathbb{R}/\mathbb{Z}\rightarrow \{1\}$$ is a short exact sequence. Let $U:=U_1\times U_2$, since $\mathbb{Z}$ is dense in $\Delta_1$ and $\Delta_2$ one can show that $U$ is connected.
\subsection{The $G$-action}
Let $G=G_1\oplus G_2$ where $G_{i} = \bigoplus_{p\in\mP_i}\mathbb{F}_p$ we construct a homomorphism $\sigma:G\rightarrow U$ such that the $G$-system $(U,G)$ is an ergodic Kronecker system:\\ 
Let $i,j$ such that $\{i,j\}=\{1,2\}$, let $g$ be a generator of order $p\in \mP_i$ of $G_{i}$. We denote by $v_g\in \Delta_j$ the unique $p$'th root of the element $(\omega_{p_1},\omega_{p_2},...)\in\Delta_j$. Such root exists since any prime in $\mP_i$ does not divide the supernatural order of $\Delta_j$. We let $\sigma_i(g)$ be the element $(\frac{1}{p},v_g)\in U_j$ (so $G_i$ acts on $U_j$) this is an element of order $p$ in $U_j$ hence $\sigma$ extends uniquely to a homomorphism.\\
In particular we have a $G$-action on $U$ by $T_gu = \sigma(g)\cdot u$ where $\sigma(g)=(\sigma_1(g),\sigma_2(g))$. We claim that this action is ergodic, equivalently that the image of $G$ under $\sigma$ is dense (See \cite[Corollary 3.8]{Zim}). Let $\pi:\mathbb{R}\times\Delta_1 \times\mathbb{R}\times \Delta_2\rightarrow U$ be the quotient map. This is a continuous and open map onto $U$.
We consider a general open set of the form $W_1\times v_1\cdot V_1\times W_2 \times v_2\cdot V_2$ where $W_1,W_2$ are balls in $\mathbb{R}$, $v_iV_i$ is a co-set of some open subgroup $V_i\leq \Delta_i$ for $i=1,2$ (by Proposition \ref{opensubgroup} every open subset contains a set of this form). Rotating by an element in $\mathbb{Z}$ we may assume that $W_1,W_2$ intersects non-trivially with $(0,1)$. It is enough to show that $\pi^{-1}(\sigma(G))$ intersect with this set. Let $s,t$ be the size of $\Delta_1/V_1$ and $\Delta_2/V_2$ respectively. Let $n$ be a sufficiently large number depending only on $W_1,W_2,s,t$ that we will choose later. Let $g_1\in G_1$ and $g_2\in G_2$ be two generators of orders $p_1,p_2$ for $p_1,p_2>n$.
Recall $v_{g_1},v_{g_2}$ from the definition of $\sigma$. $v_{g_i}^{p_i}$ is a generator of $\Delta_i$ and therefore so is $v_{g_i}$. Since $\Delta_1/V_1,\Delta_2/V_2$ are finite we can find powers $m_i,m_{i}'\leq \max\{s,t\}$ such that $v_{g_i}^{m_i}\in v_i\cdot V_i$ and $v_{g_i}^{m_{i}'}\in V_i$.
Hence for every $k$ we have that $v_{g_i}^{m_i+km_{i}'}$ is in $v_i\cdot V_i$. Since $m_i'$ depend only on $s,t$, for $n$ sufficiently large one of the elements in $\{\frac{m_i+km_i'}{p_i}:k\in\mathbb{N}\}$ must intersect with $W_i\cap (0,1)$. Let $g=((m_1+km_1')g_1,(m_2+km_2')g_2)$ then $\sigma(g)$ is an element in the image of $W_1\times v_1\cdot V_1 \times W_2\times v_2 V_2$ under $\pi$. Thus, $\sigma(G)$ intersects with any open subset in $U$, therefore is dense.
\subsection{A cocycle that is not cohomologous to a phase polynomial}
For every $g\in G_1\oplus G_2$ choose any $\alpha_g=(a_g,b_g)\in\mathbb{R}^2$ such that $a_g \text{ mod 1}$ is equal to the first coordinate of $\sigma(g)$ in $U_1$ and $b_g \text{ mod 1}$ to the first coordinate in $U_2$. In particular if $g\in G_1$ then $a_g=0$ and if $g\in G_2$ then $b_g=0$.\\
Let $x,y\in\mathbb{R}^2$ we define a bilinear form $\phi:\mathbb{R}^2\times \mathbb{R}^2\rightarrow S^1$ by the formula $\phi(x,y)=e(x_1y_2-x_2y_1)$ where $e(a)=e^{2\pi i a}$, and let $\tilde{f}(x) = \phi(x,\lfloor x\rfloor)$ where $\lfloor x\rfloor$ is the integer part of $x$ coordinate-wise. For every $k\in\mathbb{Z}^2$ we have,
\begin{equation}\label{eq1}
\tilde{f}(x+k)=\tilde{f}(x)\cdot \phi(x,k)
\end{equation}
This gives a rise to a function $f:G\times \mathbb{R}^2\rightarrow S^1$ which is given by $$f(g,x) = \frac{\tilde{f}(x+\alpha_g)}{\tilde{f}(x)}\cdot \overline{ \phi(\alpha_g,x)}$$ Equation (\ref{eq1}) implies that $f$ is well defined on $(\mathbb{R}/\mathbb{Z})^2$. We note that $(\mathbb{R}/\mathbb{Z})^2\cong U/(\Delta_1\times \Delta_2)$ and so we can think of $f$ as a function on $U$ which is invariant under multiplication by an element in $\Delta_1\times \Delta_2$.
\begin{lem} \label{fCL}
	The vertical derivatives of $f$ are cohomologous to a constant, these constants form a character of $\mathbb{R}^2$.
\end{lem}
\begin{proof}
	Let $t\in\mathbb{R}^2$ and let $F_t(x) :=\frac{\tilde{f}(x+t)}{\tilde{f}(x)}\cdot \overline{ \phi(t,x)}$, equation (\ref{eq1}) implies that $F_t$ is well defined on $(\mathbb{R}/\mathbb{Z})^2$. A direct computation shows that $\Delta_t f = \overline{ \phi(\alpha_g,t)}\phi(t,\alpha_g)\cdot \Delta_{\alpha_g} F_t$. The constant $\overline{ \phi(\alpha_g,t)}\phi(t,\alpha_g)$ gives a rise to a character of $\mathbb{R}^2$, $\chi(g,x):= \phi(\alpha_g,x)\cdot \overline{\phi(x,\alpha_g)}$ such that $\Delta_t (f\cdot \chi) = \Delta F_t$.
\end{proof}
Now we extend $\chi$ to a character of $U$ we define a map $\varphi:G\times \Delta_1\times \Delta_2 \rightarrow S^1$ by the formula $\varphi(g,v_1,v_2):=(v_2(g))^2\cdot (v_1(g))^{-2}$ where $v_1,v_2$ are identified with the corresponding elements in $\hat G\cong \Delta_1\times \Delta_2$. It is not hard to see that $\chi\cdot\varphi:G\times \mathbb{R}^2\times \Delta_1\times\Delta_2\rightarrow S^1$ is well defined as a function (homomorphism) on $U$ (i.e. it is invariant under multiplication by an element in $\mathbb{Z}$).\\
$f\cdot \chi\cdot\varphi$ is not a cocycle, but it satisfies the following facts:
\begin{thm} \label{cocycle}
	$f\cdot\chi\cdot \varphi$ is a line cocycle and $\Delta_h f\cdot \chi\cdot \varphi (g,\cdot) = \Delta_g f\cdot \chi\cdot \varphi(h,\cdot)$
\end{thm}
Assuming this theorem, one can define a cocycle $\rho:G\times U\rightarrow S^1$ to be the unique cocycle which equals to $f\cdot\chi\cdot\varphi$ on the generators of $G$. That is,

\begin{equation} \label{rho}\rho(g,u) = \prod_{i=1}^\infty T_{g_1}T_{g_2}...T_{g_{i-1}} \prod_{k=0}^{g_i} f\cdot\chi\cdot \varphi(e_i,T_{e_i}^k)
\end{equation}
We note that the infinite product is well defined because it is finite for every given $g\in G$.

\begin{proof}[Proof of Theorem \ref{cocycle}]
	Let $g\in G$ be an element of order $n$
	$$\prod_{k=0}^{n-1} f\cdot\chi\cdot\varphi(g,T_g^k(x,v)) = \prod_{k=0}^{n-1} \frac{\tilde{f}(x+(k+1)\alpha_g)}{\tilde{f}(x+k\alpha_g)}\overline{\phi(x+k\alpha_g,\alpha_g)} \cdot \varphi(g,v+v_g^k)$$
	
	The first term is a telescoping series. Since $\phi(\alpha_g,\alpha_g)=1$ the second term is $\overline{\phi(nx+\binom{n}{2}\alpha_g,\alpha_g)}=\overline{ \phi(x, n\alpha_g)}$, the third term equals to $\varphi(g,v)^n \cdot \varphi(g,v_g) ^{\binom{n}{2}}$ which is trivial since $n$ divides $2\cdot \binom{n}{2}$. We conclude that
	$$\prod_{k=0}^{n-1} f\cdot\chi\cdot\varphi(g,T_g^k(x,v)) = \frac{\tilde{f}(x+n\alpha_g)}{\tilde{f}(x)}\overline{ \phi(x,n\alpha_g)} = 1$$
	In other words $f\cdot\chi\cdot\varphi$ is a line-cocycle.\\
	
	We now prove the second property, since $\Delta_t f \cdot \chi (g,x)=  \Delta_{\alpha_g} F_t(x)$ one has that $$\Delta_{\alpha_h} f\cdot \chi(g,x) = \Delta_{\alpha_g} F_{\alpha_h}(x) = \Delta_{\alpha_g} f(h,x) $$ Let $\sigma(g) = (\alpha_g,v_g)$ where $v_g\in \Delta_1\times \Delta_2$, it is enough to show that
	
	$$1 = \Delta_{v_h} \overline{\varphi(g,v)} \Delta_{\alpha_g} \chi(h,x)\cdot \Delta_{v_g}\varphi(h,v)$$
	Since $\chi(h,\alpha_g)\cdot \varphi(h,v_g)$ is a homomorphism in $g$ one has that the product above is a homomorphism in $g$ and a homomorphism in $h$. Hence it is enough to check equality in the case where $h$ and $g$ are generators of order $p_h$ and $p_g$ respectively. If $h=g$ then the claim is obvious since $\chi(h,\alpha_h)=1$ while if $h\not = g$ then the term above is of order $p_h$ and of order $p_g$ and so is trivial.
\end{proof}
Let $\rho$ be as in (\ref{rho}) we have,
\begin{thm}
	The system $X=U\times_{\rho} S^1$ is an ergodic system of order $<3$ that is not Abramov.
\end{thm}
\begin{proof}
	First we claim that the system is ergodic. It is enough to show that $\rho$ is not $(G,U,S^1)$-cohomologous to a cocycle taking values in some proper subgroup of $S^1$ (See \cite[Corollary 3.8]{Zim}). If by contradiction it is then there exists an $n\in\mathbb{N}$ such that $\rho^n$ is a $(G,U,S^1)$-coboundary. Then the Conze-Lesigne equation gives $$\Delta_t \rho = \lambda_t\cdot \Delta F_t$$ for every $t\in U$ where $\lambda_t:G\rightarrow S^1$ is a homomorphism such that $\lambda_t^n$ is a coboundary. From the cocycle identity we conclude that $\lambda_{t^n}$ is a coboundary. Since $U$ is connected, every $u\in U$ can be written as $t^n$. Lemma \ref{cob:lem} then implies that $\rho$ is cohomologous to constant. This is a contradiction and so $X$ is ergodic.\\
	Now we prove that $X$ is of order $<3$, equivalently show that $\rho$ is of type $<2$. By the Conze-Lesigne equation with $t=x'\cdot x^{-1}$ we have $$\frac{\rho(x)}{\rho(x')} = \lambda_{x-x'} \Delta F_{x-x'}(x)$$ Since the action of $G$ on $U$ is given by a homomorphism we have that the map $(x,x')\mapsto \Delta F_{x-x'}(x)$ is a derivative of $G(x,x'):=F_{x-x'}(x)$ in $U\times U$ therefore is a coboundary, hence $(x,x')\mapsto \frac{\rho(x)}{\rho(x')}$ is cohomologous to $\lambda_{x-x'}$ which is invariant under the action of $G$ and so is of type $<1$. We conclude that $\rho$ is of type $<2$. Since $X$ is connected there are only phase polynomials of degree $<2$. Such phase polynomials are measurable with respect to the Kronecker system $U$. In particular the map $(u,x)\mapsto x$ is orthogonal to every phase polynomial. Therefore $X$ is not Abramov as required.
\end{proof}
\appendix
\section{Topological groups and measurable homomorphisms}
In this section we survey some results on topological groups. We concentrate on polish groups and measurable homomorphisms on them and on totally disconnected groups.
\subsection{Homomorphisms of polish groups}

\begin{defn}
	We say that a topological group $G$ is a polish group if it is separable and completely metrizable. If $G$ is compact this is equivalent to the existence of some invariant metric on $G$ (i.e. a metric such that $d(x,y)=d(gx,gy)$).
\end{defn}

Every topological group is a measure space with respect to the Borel $\sigma$-algebra. It is well known that every locally compact abelian group $G$ admits a unique (up to scalar multiplication) invariant measure $\mu$. This measure is inner and outer regular and it assigns finite measure for compact subsets. In particular in the case where $G$ is compact we can normalize so that $\mu(G)=1$. The existence of such measures leads to many fruitful corollaries.
\begin{prop} [A.Weil]  \cite[Lemma 2.3]{Ch}\label{A.Weil}
	Let $G$ be a locally compact polish group and let $A\subseteq G$ be a measurable subset of positive measure then $AA^{-1}$ contains an open neighborhood of the identity.
\end{prop}
This implies the following useful Proposition.
\begin{prop}\label{countableindex}
	Let $G$ be a locally compact polish abelian group and let $H$ be a Borel subgroup of at most countable index. Then $H$ is open.
\end{prop}
\begin{proof}
	Let $\mu$ be the Haar measure on $G$. Since $H$ has countable index there exists $g_1,g_2,...$ such that $G=\bigsqcup_{i=1}^\infty g_i H$, in particular it follows that $$0<\mu(G)= \sum_{i=1}^\infty \mu(g_i H)$$ since the measure is invariant the right hand side is an infinite sum of $\mu(H)$. This is only possible if the measure of $H$ is positive (Note that if $G$ is compact this also implies that the measure is finite). Now, by Proposition \ref{A.Weil} we have that $H-H$ contains an open neighborhood $U$ of the identity. Since $H-H\subseteq H$ we have that $H=\bigcup_{h\in H} hU$ and so is open.
\end{proof}
As a corollary we have
\begin{cor}\label{openker} Let $G$ be a locally compact abelian polish group, let $L$ be a locally compact abelian group of at most countable cardinality. Then any measurable homomorphism $\varphi:G\rightarrow L$ factors through an open subgroup of $G$.
\end{cor}
\begin{proof}
	The kernel of $\varphi$ is a Borel subgroup of at most countable index. Therefore the claim follows from the previous proposition.
\end{proof}

Another important Corollary of Proposition \ref{A.Weil} is the following automatic continuity lemma.

\begin{lem} [Automatic continuity of measurable homomorphisms] \cite[Theorem 2.2]{Ch}\label{AC:lem}
	Any measurable homomorphism from a locally compact polish group into a polish group is continuous.
\end{lem} 

\subsection{Totally disconnected groups}
\begin{defn} \cite[Exercise E8.6]{HM} Let $X$ be a locally compact Hausdorff space. Then the following are equivalent
	\begin{itemize}
		\item{Every connected component in $X$ is a Singleton}
		\item {$X$ has a basis consisting of open closed sets}
	\end{itemize}
	We say that $X$ is totally disconnected if one of the above is satisfied.
\end{defn}

In this section we will be interested in compact (Hausdorff) totally disconnected groups. These groups are also called pro-finite groups, in fact one can show that every such group is an inverse limit of finite groups (see \cite[Proposition 1.1.3]{profinite}).

\begin{prop}[Open neighborhood of the identity of compact Hausdorff totally disconnected groups contains an open subgroup]\label{opensubgroup}
	Let $G$ be a compact Hausdorff totally disconnected group. Let $1\in U\subseteq G$ be an open neighborhood of the identity, then $U$ contains an open subgroup of $G$.
\end{prop}
The proof of this Proposition can be found in \cite[Proposition 1.1.3]{profinite}. As a corollary we have,
\begin{cor}[The dual of totally disconnected group is a torsion group]  \label{chartdg} Let $G$ be a compact abelian totally disconnected group and let $\chi:G\rightarrow S^1$ be a continuous character. Then the image of $\chi$ is finite.
\end{cor}
\begin{proof}
	Choose an open neighborhood of the identity $U$ in $S^1$ that contains no non-trivial subgroups. Then $\chi^{-1}(U)$ is an open neighborhood of $G$. Now, let $H$ be an open subgroup such that $H\subseteq \chi^{-1}(U)$. It follows that $\chi(H)$ is trivial and so $\chi$ factors through $G/H$ which is finite.
\end{proof}
We note that the other direction also holds, but we do not use this fact here.

Since compact totally disconnected groups are pro-finite groups. Some of the theory of finite groups can be generalized to these groups. For example we have,
\begin{prop}[Sylow Theorem]\cite[Corollary 8.8]{HM} \label{Sylow}
	A compact abelian group is totally disconnected if and only if it is a direct
	product of $p$-groups.
\end{prop}
We also need the following structure theorem for torsion groups (of bounded torsion).
\begin{thm}[Structure theorem for abelian groups of bounded torsion]\label{torsion}\cite[Chapter 5, Theorem 18]{M}
	Let $G$ be a compact abelian group, and let $n$ be an integer such that $g^n=1_G$ for all $g\in G$. Then $G$ is topologically and algebraically isomorphic to $\prod_{i=1}^\infty C_{m_i}$ such that for all $i$, the integer $m_i$ divides $n$.
\end{thm}
One way to generate totally disconnected groups is to begin with an arbitrary compact abelian group and quotient it out by it's connected component.
\begin{lem}\label{connectedcomponent}
	Let $G$ be a compact abelian group and let $G_0$ be the connected component of the identity. Since the multiplication and the inversion maps are continuous one has that $G_0$ is a subgroup of $G$. We have
	\begin{itemize}
		\item{$G_0$ has no non-trivial open subgroups}
		\item {Every open subgroup of $G$ contains $G_0$}
		\item {$G/G_0$ equipped with the quotient topology is totally disconnected compact group}
	\end{itemize}
\end{lem}
As a Corollary we have,
\begin{prop}[Quotient of profinite group is a profinite group] \label{Quotient} 
	Let $G$ be a profinite group, let $N$ be a subgroup of $G$. Then $G/N$ with the induced topology is a profinite group.
\end{prop}
\begin{proof}
	Let $C$ denote the connected component of the identity in $G/N$. Let $x\in C$, let $\pi:G\rightarrow G/N$ be the projection map. Then $\pi^{-1}(\{x\})$ is a closed subset of $G$. If by contradiction $x\not = 1$ then by Proposition \ref{opensubgroup} we have that the complement contains an open subgroup $V$. Quotient homomorphisms are open and so by Lemma \ref{connectedcomponent}, $\pi(V)$ contains the connected component of $G/N$ which is absurd.
\end{proof}
\subsection{Lie groups}
\begin{defn}
	A topological group $G$ is said to be a Lie group, if as a topological space it is a finite dimensional, differentiable manifold over $\mathbb{R}$ and the multiplication and the inversion maps are smooth.
\end{defn}
A compact abelian group is a Lie group if and only if it's Pontryagin dual is finitely generated. The structure theorem for finitely generated abelian group then implies
\begin{thm}[Structure Theorem for compact abelian Lie groups]\cite[Theorem 5.2]{S} \label{structureLieGroups} A compact abelian group $G$ is a Lie group if and only if $n\in\mathbb{N}$ such that $G\cong (S^1)^n\times C_k$ where $C_k$ is some finite group with discrete topology.
\end{thm} \label{approxLieGroups}
Thus, compact abelian Lie groups are very simple. Fortunately it turns out that every group can be approximated by compact abelian Lie groups
\begin{thm}\cite[Corollary 8.18]{HM} \label{GY:thm}
	Let $G$ be a compact abelian group and let $U$ be a neighborhood of the identity in $G$. Then $U$ contains a subgroup $N$ such that $G/N$ is a Lie group.
\end{thm}

From the above we conclude
\begin{prop}\label{connilabel:prop}
	If $G$ is a compact metric connected $k$-step nilpotent group. Then $G$ is abelian
\end{prop}
\begin{proof}
Let $\{U_i\}$ be a countable neighborhood of $e$ with $\bigcap_{i\in\mathbb{N}} U_i = \{e\}$. For each choose a group $N_i$ as in the theorem above. The group $G/N_i$ is a compact connected $k$-step nilpotent Lie group and therefore is a Torus. We conclude that $[G,G]\subseteq N_i$ for all $i$ and so is trivial. Hence $G$ abelian.
\end{proof}
\section{Some results about phase polynomials}
\begin{prop}  [Values of phase polynomial cocycles]\label{PPC} Let $X$ be an ergodic $G$-system, let $d\geq 0$ and $q:G\times X\rightarrow S^1$ be a phase polynomial of degree $<d$ that is also a cocycle. Then for $g\in G$, $q(g,\cdot)$ takes values in $C_m$ where $m$ is the order of $g$ to the power of $d$.
\end{prop} 
\begin{proof}
	We prove by induction on $d$, for $d=0$ the claim is trivial. 
	Assuming for smaller values of $d$, let $q:G\times X\rightarrow S^1$ be a phase polynomial of degree $<d$ and fix $g\in G$ of order $n$, then the cocycle identity implies that $$1=q(ng,x)=\prod_{k=0}^{n-1}q(g,T_{kg}x)$$
	We write $q(g,T_{kg}x)=q(g,x)\cdot \Delta_{kg} q(g,x)$ we have that $q(g,x)^n \cdot \prod_{k=0}^{n-1}\Delta_{kg} q(g,x)=1$. By induction hypothesis we have that $\prod_{k=0}^{n-1}\Delta_{kg} q(g,x)$ is an element of $C_{n^{d-1}}$, hence $q(g,x)\in C_{n^d}$. This completes the proof.
\end{proof}

We need the following version of \cite[Lemma B.5 (i)]{Berg& tao & ziegler}
\begin{lem}  [Vertical derivatives of phase polynomials are phase polynomials of smaller degree]\label{vdif:lem}
	Let $X$ be an ergodic $G$-system. Let $U$ be a compact abelian group acting freely on $X$ which commutes with the action of $G$. Let $P:X\rightarrow S^1$ be a phase polynomial of degree $<d$ for some integer $d\geq 1$. Then $\Delta_u P$ is a phase polynomial of degree $<d-1$ for every $u\in U$.
\end{lem}
\begin{proof} 
	By induction on $d$, for $d=1$ ergodicity implies that  $P$ is a constant and so $\Delta_u P =1$. Let $d\geq 2$ and assume inductively that the claim is true for $d-1$. Given a phase polynomial $P:X\rightarrow S^1$ of degree $<d$, we have that $\Delta P : G\times X\rightarrow S^1$ is a phase polynomial of degree $<d-1$. By the induction hypothesis we conclude that $\Delta_u \Delta P$ is a phase polynomial of degree $<d-2$. As the action of $U$ commutes with the action of $G$ we have that $\Delta \Delta_u P$ is a phase polynomial of degree $<d-2$. It follows that $\Delta_u P$ is a phase polynomial of degree $<d-1$ as desired.
\end{proof}
Proposition \ref{PPC} and Lemma \ref{vdif:lem} implies 
\begin{cor} \label{ker:cor} 
	Let $X$ be an ergodic $G$-system, let $U$ be a compact abelian group acting freely on $X$ and commuting with the action of $G$. Suppose that there exists a measurable map $u\mapsto f_u$ from $U$ to $P_{<d}(X,S^1)$ which satisfies the cocycle identity (i.e. $f_{uv}=f_u V_u f_v$) for all $u,v\in U$. Then there exists an open subgroup $V$ of $U$ such that $f_v\in P_{<1}(X,S^1)$ for all $v\in V$.
\end{cor}
\begin{proof}
	We prove by induction on $d$, for $d=1$ we can take $V=U$. Let $d>1$ and assume the claim is true for smaller values of $d$. Let $u\mapsto f_u$ be a map from $U$ to $P_{<d}(X,S^1)$. The cocycle identity implies that $f_{uv}=f_u f_v \cdot \Delta_u f_v$. Applying Lemma \ref{vdif:lem} we have that $\Delta_u f_v\in P_{<d-1}(X,S^1)$ and so after quotienting out $P_{<d-1}(X,S^1)$ we have that the map $U\rightarrow P_{<d}(X,S^1)/P_{<d-1}(X,S^1)$ sending $u$ to the equivalent class of $f_u$ is a homomorphism. Since $d>1$, Lemma \ref{sep:lem} (Separation Lemma) implies that $P_{<d-1}(X,S^1)$ has at most countable index in $P_{<d}(X,S^1)$. Corollary \ref{openker} implies that the kernel, $U'$ is an open subgroup. \\
	We conclude that $f_{u'}\in P_{<d-1}(X,S^1)$ for all $u'\in U'$, and so the induction hypothesis implies that there exists an open subgroup $V$ of $U'$ such that $f_v\in P_{<1}(X,S^1)$ for all $v\in V$. As $V$ is open in $U'$ and $U'$ is open in $U$ we have that $V$ is open in $U$.
\end{proof}
we also need the following Lemma from \cite{Berg& tao & ziegler}
\begin{lem}   [Composition of polynomials is again polynomial] \cite[Lemma B.5 (iii)]{Berg& tao & ziegler}\label{B.5} Let $U,V$ be two abelian groups and $X=Y\times_{\rho} U$ be an ergodic extension of a $G$-system $Y$, by a phase polynomial cocycle $\rho:G\times Y\rightarrow U$ of degree $<k$ for some $k\geq 1$. Then, If $p:X\rightarrow V$ is a phase polynomial of degree $<d$, and $v_1:X\rightarrow U,v_2:X\rightarrow U,...,v_j:X\rightarrow U$ are a collection of phase polynomials of degree $<d_1,<d_2,...,<d_j$ and $s:X\rightarrow U$ is a phase polynomial of degree $<d'$. Then the function $P:X\rightarrow V$ given by the formula
	$$p(y,u)=(\Delta_{v_1(y,u)}...\Delta_{v_j(y,u)}p)(y,s(y,u))$$
	is a phase polynomial of degree 		 $<O_{d,j,d1,...,dj,d',k}(1)$.
\end{lem}
\begin{prop} [Phase polynomial are invariant under connected components] \label{pinv:prop} (see \cite[Lemma 2.1]{BTZ})  Let $X$ be an ergodic $G$-system of order $<k$, let $U$ be a compact abelian connected group acting freely on $X$ (not necessarily commuting with the $G$-action). Let $P:G\times X\rightarrow S^1$ be a phase polynomial of degree $<d$ such that for every $g\in G$ there exists $M_g\in\mathbb{N}$ such that $P(g,\cdot)$ takes at most $M_g$ values. (e.g $P$ is a phase polynomial cocycle). Then $P$ is invariant under the action of $U$.
\end{prop}
\begin{proof}
	Fix $g\in G$ and consider the map $u\mapsto \Delta_u P(g,\cdot)$. This is a measurable map from $U$ to $P_{<d}(X,S^1)$, we have that $\Delta_u P$ converges in measure to the constant $1$ as $u$ converges to the identity in $U$. From Lemma \ref{sep:lem} we conclude that $\Delta_u P$ must be almost everywhere constant for $u$ close to the identity. From the cocycle identity we have that the subset $U' = \{u\in U : \Delta_u P \text{ is a constant}\}$ is an open subgroup of $U$. As $U$ is connected we conclude that $U'=U$ and so $\Delta_u P(g,\cdot) = \chi(u)$ for a character $\chi:U\rightarrow S^1$. Now, since $U$ is connected and $\chi$ is continuous we have that the image of $\chi$ is either trivial or is $S^1$. But, the latter contradicts the assumption that $P(g,\cdot)$ takes finitely many values, it follows that $\Delta_u P(g,\cdot)=1$ for every $u\in U$. In other words, $P$ is invariant under the action of $U$.
\end{proof}
\begin{rem}
	The group $S^1$ in the Proposition can be replaced by any compact abelian group using Pontryagin duality. If $P:G\times X\rightarrow V$ is a phase polynomial for some compact abelian group $V$, then for every $\chi\in\hat V$ we have that $\chi\circ P:G\times X\rightarrow S^1$ is a phase polynomial of the same degree. By applying Proposition \ref{pinv:prop} we have that $\chi(\Delta_uP)=1$ for every $u\in U$. As the characters separates points this would imply that $\Delta_u P = 1$, hence $P$ is invariant under $U$.
\end{rem}
\begin{prop} [Phase polynomials on totally disconnected systems take finitely many values]\label{TDPV:prop} 
	Let $X$ be an ergodic totally disconnected $G$-system of order $<k$ (see Definition \ref{TD:def}). Let $P:X\rightarrow S^1$ be a phase polynomial of degree $<d$, then up to a constant multiple, $P$ takes values in a finite subgroup of $S^1$.
\end{prop}
\begin{proof}
	We induct on $k$, for $k=1$ the claim is trivial. Let $k\geq 2$ and assume the claim has already been proven for $k-1$. Let $X$ be as in the Proposition, then by Proposition \ref{abelext:prop} we can write $X=Z_{<k-1}(X)\times_\rho U$.\\
	Consider the map $u\mapsto \Delta_u P$. Clearly, $u\mapsto \Delta_u P$ satisfies the cocycle identity and so Corollary \ref{ker:cor} implies that there exists an open subgroup $V$ s.t $\Delta_u P\in P_{<1}(X,S^1)$ for every $v\in V$. Ergodicity implies that $\Delta_u P$ is a constant in $S^1$. The induced map $V\rightarrow S^1$ sending $v$ to the constant $\Delta_v P$ is a homomorphism. As $V$ is totally disconnected the kernel $U'$ is an open subgroup which satisfies that $\Delta_u P =1$ for every $u\in U'$.\\
	Now, let $u\in U$. As $U/U'$ is a finite group, there exists $m\in\mathbb{N}$ such that $u^m\in U'$. We follow the argument in the proof of Corollary \ref{ker:cor}, but instead of reducing to an open subgroup we take a power. As in Corollary \ref{ker:cor} we have that $U\rightarrow P_{<d-1}(X,S^1)/P_{<d-2}(X,S^1)$ sending $u$ to the equivalent class of $\Delta_u P$ is a homomorphism. As $\Delta_{u^m}P=1$ for all $u\in U$ we conclude that $\Delta_u P^m$ is a phase polynomial of degree $<d-2$. Iterating this process, we conclude that $\Delta_u P^{m^{d-1}}=1$ for all $u\in U$, in other words $P^{m^{d-1}}$ is invariant under $U$. Viewing $P^{m^{d-1}}$ as a phase polynomial of degree $<d$ on $Z_{<k-1}(X)$, and applying the induction hypothesis we see that up to constant multiplication $P^{m^{d-1}}$ takes values in some finite subgroup $H$ of $S^1$. Let $c$ be an $m^{d-1}$'th root of this constant, we have that $P/c$ takes values in the finite group $H$, as desired.
	
\end{proof}
As a Corollary we have,
\begin{thm} [Phase polynomials of degree $<d$ on totally disconnected systems take $<O_d(1)$ values on finitely many primes]\label{TDPV:thm}
	Let $X$ be an ergodic totally disconnected $G$-system of order $<k$ (see Definition \ref{TD:def}), let $F:X\rightarrow S^1$ be a phase polynomial of degree $<d$. Then up to constant multiplication, $F$ takes values in the group $C_m$ where $m=p_1^{l_1}\cdot...\cdot p_n^{l_n}$ for some $n\in\mathbb{N}$, distinct primes $p_1,...,p_n$ and $l_1,...,l_n=O_{d}(1)$.
\end{thm}
\begin{proof}
	By Proposition \ref{TDPV:prop} we have that up to constant multiplication $F$ takes values in $C_\alpha$ for some finite $\alpha\in\mathbb{N}$. Let $q=\Delta F$, then $q$ is a phase polynomial of degree $<d-1$ which is also a cocycle and it takes values in $C_\alpha$. Let $n$ be such that $\alpha<p_n$ where $p_n$ is the $n$'th prime and write $G=G_n\oplus G'$ where $G_n=\bigoplus_{p\in P,p<p_n}\mathbb{F}_p$ and $G'$ it's complement. From Proposition \ref{PPC} and the fact that $\alpha<p_n$ we conclude that $q(g,\cdot)=1$ for all $g\in G'$.  Let $m=\prod_{i=1}^n p_i^{d-1}$, from Proposition \ref{PPC} we see that $q^m=1$ and so by ergodicity $F^m$ is constant. Let $c$ an $m$'th root for $\overline {F^m}$, we have that $c\cdot F$ takes values in $C_m$.
\end{proof}
In the next theorem we will generalize the results above to the case where $p\geq k$.
\begin{thm} \label{HTDPV:thm} 
	Let $X$ be an ergodic $G$-system, let $F:X\rightarrow S^1$ be a phase polynomial of degree $<k$, and let $p$ be a prime such that $p\geq k$. Suppose now that $F$ takes values in $C_{p^m}$ for some $m\in\mathbb{N}$, then up to constant multiplication $F$ takes values in $C_p$
\end{thm}
\begin{proof}
	Consider $\Delta F:G\times X\rightarrow C_{p^m}$, this is a phase polynomial of degree $<k$. Write $G=G_p\oplus G'$ where $G_p$ is the $p$-torsion subgroup of $G$, then for every $g\in G'$ we have by Proposition \ref{PPC} that $\Delta_g F$ takes values in $C_n$ for some $n$ co-prime to $p$, as $\Delta F$ takes values in $C_{p^m}$ we conclude that $\Delta_g F=1$ for all $g\in G'$. We claim that $\Delta_g F^p = 1$ for all $g\in G_p$ (Using ergodicity and the cocycle identity, this would imply that $F^p$ is constant).\\
	Now we argue as in \cite[Lemma D.3 (i)]{Berg& tao & ziegler}.
	
	Taking logarithm it is sufficient to show that if $F:X\rightarrow \mathbb{Z}/{p^m}\mathbb{Z}$ is an (Additive) polynomial of degree $<k$ then $pF$ is a constant.
	
	Let $g\in G_p$, then $T_g^p F = F$ write $T_g = 1+\Delta^+_g$ where $\Delta^+_g f (x)= f(T_g(x))-f(x)$ is the additive derivative. We conclude using the binomial formula that $\sum_{i=0}^p \binom{p}{i}(\Delta_g^+)^iF=F$. Since $F$ has degree $k$ and $k\leq p$ we have that $(\Delta_g^+)^pF=0$. We conclude that $$p\Delta_g^+F+\binom{p}{2} (\Delta_g^+)^2 F+...+p(\Delta_g^+)^{p-1} F=0$$
	
	which we rewrite as
	
	$$\left(1+\frac{p-1}{2}\Delta_g^++...(\Delta_g^+)^{p-2}\right)(\Delta_g^+) pF =0$$
	Inverting the expression in the bracket using Neumann series (and using the fact that $(\Delta_g^+)^{p-1}$ annihilates $(\Delta_g^+) pF$) we conclude that $\Delta_g^+ pF=0$ for any $g\in G_p$, as we said before $\Delta_g^+ pF=0$ for every $g\in G'$ as well. Using the cocycle identity we conclude that $\Delta_g^+ pF=0$ for every $g\in G$ hence ergodicity implies that $pF$ is a constant.
	
\end{proof}
\begin{rem}
	We can generalize the result of Theorem \ref{HTDPV:thm} for phase polynomials which takes finitely many values. For if $F:X\rightarrow S^1$ takes values in some finite subgroup $H$ of $S^1$, we can write $H= C_{{p_1}^{l_1}}\times...\times C_{{p_N}^{l_N}}$. Composing $F$ with each of the coordinate maps $\pi_1,...,\pi_N$ yields a polynomial as in the previous theorem, hence if $p_1,...,p_N$ are sufficiently large then $F$ takes values in $C_{p_1}\times...\times C_{p_N}$.

\end{rem}
We denote by $P^C_{<k}(G,X,S^1)$ the subgroup of $P_{<k}(G,X,S^1)$ of phase polynomials of degree $<k$ that are also cocycles. Clearly $P^C_{<k}(G,X,S^1)$ is a closed subgroup. 
\begin{cor} [Compact subgroups of $P^C_{<k}(G,X,S^1)$] Let $G=\bigoplus_{p\in \mathcal{P}}\mathbb{F}_p$ for some multiset of primes. Let $X$ be an ergodic $G$-system and let $k<\min\mathcal{P}$. Then any compact subgroup $K$ of $P^C_{<k}(G,X,S^1)$ is a direct product of copies of $C_p$ for some primes in $\mathcal{P}$. 
\end{cor}
\begin{proof}
	$K$ is a closed subgroup of $P^C_{<k}(G,X,S^1)$ which is closed in the direct product $\prod_{p\in\mathcal{P}} P^C_{<k}(\mathbb{F}_p^\omega,X,S^1)$. From the theorem above, since $k<p$ we have that each of the $P^C_{<k}(\mathbb{F}_p,X,S^1)$ is a $p$-torsion group. In particular $K$ is totally disconnected and the $p$-sylow subgroup of $K$ is $p$-torsion. By Theorem \ref{torsion} we have that $K_p$ is a direct product of $C_p$. The proof is now complete by Theorem \ref{Sylow}.
\end{proof}
We generalize another result from \cite[Lemma D.3]{Berg& tao & ziegler}

\begin{prop} [line cocycles] \label{linecocycle:prop} 
	Let $X$ be an ergodic $G$-system. Let $F:X\rightarrow S^1$ be a phase polynomials of degree $<k$ and suppose that $F$ takes values in $C_n$ for some $n=p_1\cdot...\cdot p_j$ where $k< p_1,p_2,...,p_j$. Let $g\in G$, then $\prod_{t=0}^{n-1} T_g^t F = 1$. 
\end{prop}
\begin{proof}
	Write $C_n = C_{p_1}\times...\times C_{p_j}$, and let $\pi_i : C_n\rightarrow C_{p_i}$ be the projection map.  We show that $\prod_{t=0}^{n-1} T_g^t F_i=1$ where $F_i = \pi_i \circ F$.
	
	First we decompose $G$ as $G=G_{p_i}\oplus G'$ where $G_{p_i} = \{g\in G : p_ig=0\}$ and $G'$ it's complement. Taking logarithm, it is enough to show that every polynomial $F:X\rightarrow \mathbb{R}/\mathbb{Z}$ with $nF=0$ satisfies that $\sum_{t=0}^{n-1} T_g^t F = 0$. 
	
	Given $g\in G$ we can write $g=g_i+g'$ where $g_i\in G_{p_i}$ and $g'\in G'$. Since $F_i$ is a phase polynomial takes values in $\mathbb{Z}/{p_i}\mathbb{Z}$, by Proposition \ref{PPC} it is invariant under $T_{g'}$ we conclude that  $\sum_{t=0}^{n-1} T_g^t F=\sum_{t=0}^{n-1} T_{g_i}^t F$. If $g_i=0$ then $\sum_{t=0}^{n-1}T_gF=nF=0$. Otherwise since $\Delta_g^+ = T_g -1$, using the binomial formula we have,
	$$\sum_{t=0}^{p_i-1}T_{g_i}^t F_i=\sum_{t=0}^{p_i-1} \binom{p_i}{t+1}(\Delta_{g_i}^+)^t F_i$$
	Since $F$ is a phase polynomial of degree $<k$, direct computation shows that so is $F_i$. Repeated application of Lemma \ref{vdif:lem} and the fact that $k<p_i$ implies that $(\Delta^+_{g_i})^{p_i-1}$ annihilates $F_i$. Since $p_i$ divides $\binom{p_i}{t+1}$ for $0\leq t\leq p_i-1$ we conclude that
	$$\sum_{t=0}^{p_i-1}T_{g_i}^t F_i=0$$
	As $g_i$ is of order $p_i$ and $p_i$ divides $n$ we have that $\sum_{t=0}^{n-1} T_{g_i}^t F_i$ is a constant multiple of $\sum_{t=0}^{p_i-1}T_{g_i}^t F_i=0$ hence trivial. Thus, for every $1\leq i\leq j$ and any $g\in G$ we have that $$\sum_{t=0}^{n-1}T_g F_i = 0$$
	we conclude that $$\sum_{t=0}^n T_g F=0$$ which completes the proof.
\end{proof}

\section{Roots of phase polynomials} \label{roots:app}
In this appendix we modify the results from appendix D in \cite{Berg& tao & ziegler}, about roots of phase polynomials.

When the group $G$ has unbounded torsion, it is not true that every phase polynomial has an $n$'th root that is also a phase polynomial (Compare with \cite[Corollary D.7]{Berg& tao & ziegler}). However, when phase polynomials take finitely many values (for example when the underline space is totally disconnected), we can use the tools developed by Bergelson Tao and Ziegler in \cite{Berg& tao & ziegler} and construct polynomial roots.

We begin with the following version of Proposition D.5 from \cite{Berg& tao & ziegler}, the proof is identical and is therefore omitted.
\begin{prop}
	Let $P:X\rightarrow \mathbb{Z}/_{p^m}\mathbb{Z}$ be an (additive) polynomial of degree $<d$, and let $\mathbb{Z}/{p^l}\mathbb{Z}$ be a cyclic group. Embed $\{0,1,...,p-1\}$ into $\mathbb{Z}/{p^l}\mathbb{Z}$ in the obvious manner. Then for any $0\leq j\leq m-1$ the map $b_j(P)$ is a polynomial of degree $<O_{l,d,p,j}(1)$. Where $b_j:\mathbb{Z}/{p^l}\mathbb{Z}\rightarrow\{0,1,...,p-1\}$ is the $j$'th digit map.
\end{prop}
Just like in \cite{Berg& tao & ziegler} this Proposition implies that functions of phase polynomials are phase polynomials. However in our case we have to add an assumption about the values of the phase polynomials. The proof is identical.

\begin{cor} [Functions of phase polynomials are phase polynomials]\label{roots}. Let $\varphi_1,...,\varphi_m$ be $m$ phase polynomials of degree $<d$ for some $d,m\geq 1$ takes values in $C_{p^d}$, let $n\geq 1$, and let $F(\varphi_1,...,\varphi_m)$ be some function of $\varphi_1,...,\varphi_m$ which takes values in the cyclic group $C_{p^n}$. Then $F(\varphi_1,...,\varphi_m)$ is a $(X,S^1)$ phase polynomial of degree $<O_{p,d,m,n}(1)$
\end{cor}
	
\address{Einstein Institute of Mathematics\\
	The Hebrew University of Jerusalem\\
	Edmond J. Safra Campus, Jerusalem, 91904, Israel \\ Or.Shalom@mail.huji.ac.il}
\end{document}